\documentclass[12pt,oneside,a4]{amsart}
\usepackage{typearea}
\usepackage{amsfonts}
\usepackage{amsmath}
\usepackage{amssymb}
\usepackage{amsthm}
\usepackage{hyperref}

\newcommand{\IC}{\mathbb{C}}

\newcommand{\IN}{\mathbb{N}}

\newcommand{\IQ}{\mathbb{Q}}
\newcommand{\IQbar}{\overline{\mathbb{Q}}}
\newcommand{\IR}{\mathbb{R}}

\newcommand{\IZ}{\mathbb{Z}}

\newcommand{\hh}{\mathcal{H}}
\newcommand{\cK}{\mathcal{K}}

\newcommand{\cO}{\mathcal{O}}

\newcounter{maincounter}

\numberwithin{maincounter}{section}
\numberwithin{equation}{section}

\newtheorem{thm}[maincounter]{Theorem}
\newtheorem*{qhyp}{Quill Hypothesis}
\newtheorem{lemma}[maincounter]{Lemma}

\newtheorem{cor}[maincounter]{Corollary}

\newtheorem{prop}[maincounter]{Proposition}

\renewcommand{\hom}{\mathrm{Hom}}

\renewcommand{\subset}{\subseteq}
\renewcommand{\supset}{\supseteq}

\newcommand{\ssm}{\smallsetminus}

\makeatletter
\def\imod#1{\allowbreak\mkern10mu({\operator@font mod}\,\,#1)}
\makeatother

\newcommand{\pco}{\mathrm{PCO}}

\newcommand{\cf}{\mathsf{C}}

\newcommand{\ecp}{\widehat{\IC}}

\begin{document}
\title[Height Lower Bounds]{Lower Bounds for the Canonical Height of a Unicritical
  Polynomial and Capacity} %Dynamical Schinzel--Zassenhaus}

\author{P. Habegger and H. Schmidt}

\address{Department of Mathematics and Computer Science, University of Basel, Spiegelgasse 1, 4051 Basel, Switzerland}
\email{philipp.habegger@unibas.ch}

\address{Department of Mathematics and Computer Science, University of Basel, Spiegelgasse 1, 4051 Basel, Switzerland}
\email{harry.schmidt@unibas.ch}
%\author{PH + HS}
\date{\today}

%%%% MSC 2020
% 11G50: Heights

% 37P15   	Dynamical systems over global ground fields

% 14G40 View Publications (1991-now) Arithmetic varieties and schemes;
% Arakelov theory; heights

%% 30C85   	Capacity and harmonic measure in the complex plane 

\maketitle

\begin{abstract}
  In a recent breakthrough, Dimitrov~\cite{dimitrov:SZ} solved the
  Schinzel--Zassenhaus Conjecture. We follow his approach and adapt it
  to certain dynamical systems arising from polynomials of the form
  $T^p+c$ where $p$ is a prime number and where the orbit of $0$ is
  finite. For example, if $p=2$, and $0$ is periodic under $T^2+c$ with
  $c\in\IR\ssm\{-2\}$, we prove a lower bound for the local canonical
  height of a wandering algebraic integer that is inversely proportional
  to the field degree. From this we are able to deduce a lower bound for
  the canonical height of a wandering point that decays like the inverse
  square of the field degree.
\end{abstract}

\section{Introduction}

Let $K$ be a field and
suppose $f\in K[T]$ has degree $\deg f\ge 2$. For
$n\in\IN=\{1,2,3,\ldots\}$ we write $f^{(n)}\in K[T]$ for the
$n$-fold iterate of $f$ and define $f^{(0)}=T$. Let $x\in K$. We call
$x$ an \textit{$f$-periodic point} if there exists $n\in\IN$ with
$f^{(n)}(x)=x$. We call $x$ an \textit{$f$-preperiodic point} if there
exists an integer $m\ge 0$ such that $f^{(m)}(x)$ is $f$-periodic.
Suppose $x$ is $f$-preperiodic, the \textit{preperiod length} of $x$ is
\begin{equation*}
  \mathrm{preper}(x) = \min \{ m\ge 0 : f^{(m)}(x) \text{ is
    $f$-periodic}\}\ge 0
\end{equation*}
and the \textit{minimum period} of $x$ is
\begin{equation*}
  \mathrm{per}(x) = \min \{ n\ge 1 : f^{(n+\mathrm{preper}(x))}(x)
  =f^{(\mathrm{preper}(x))}(x)\} \ge 1. 
\end{equation*}
An element of $K$ is called an \textit{$f$-wandering point} if it is not
$f$-preperiodic. Equivalently, $x\in K$ is an $f$-wandering point if
and only if $\{f^{(n)}(x) : n\in\IN\}$ is infinite.

Suppose for the moment that $K$ is a subfield of $\IC$. We call $f$
\textit{post-critically finite} if its \textit{post-critical set}
\begin{equation*}
  \pco^+(f) = \{f^{(m)}(x) : m\in\IN,x\in\IC, \text{ and }f'(x)=0\}.
\end{equation*}
is finite. For example, $T^2-1$ is post-critically finite as
its post-critical set equals $\{-1,0\}$. 

% \footnote{Harry: In the following we assume that K is embedded
% into the complex numbers which is slightly confusing. How about: We assume that K
% contains a splitting field of $f'$ and replace $\IC$ by $K$. Philipp:
% I think it's better to emphasize the fixed embedding $\sigma_0$. }

Let $M_{\IQ}$ denote the set consisting of the $v$-adic absolute value
on $\IQ$ for all prime numbers $v$ together with the Archimedean
absolute value.
For each prime number $v$ let $\IC_v$ denote a
completion of an algebraic closure of the valued field of $v$-adic
numbers. We also set $\IC_v=\IC$ if $v$ is Archimedean.

For a real number $t\ge 0$ we set $\log^+t = \log\max\{1,t\}$. Let
$v\in M_\IQ$ and suppose $f\in \IC_v[T]$ has degree $d\ge 2$. For all
$x\in \IC_v$ the limit
\begin{equation}
  \label{eq:deflambdafv}
  \lambda_{f,v}(x) = \lim_{n\rightarrow\infty}\frac{\log^+
    |f^{(n)}(x)|_v}{d^n}
\end{equation}
exists and satisfies $\lambda_{f,v}(f(x)) = d\lambda_{f,v}(x)$, see
Chapter 3 and in particular Theorem 3.27 or Exercise
3.24~\cite{Silverman:gtm241} for details.

If $v$ is Archimedean we often abbreviate $\lambda_{f,v} = \lambda_f$.
The local canonical height differs from $z\mapsto \log^+|z|$ by a bounded
function on $\IC$.
% , see Proposition 5.58 or the exercises of Section
% 3.5 \cite{Silverman:gtm241} for the simpler case of
% polynomials.
% \footnote{Philipp: Maybe there's a direct reference for polynomials}.
We
define
\begin{equation}
  \label{def:Cf}
  \cf(f)  = \max\left\{1,\sup_{z\in\IC} \frac{\max\{1,|z|\}}{e^{\lambda_f(z)}}\right\}<\infty. 
\end{equation}

Next we define the canonical height of an algebraic
number with respect to $f$.
Let $K$ be a number field and $f\in K[T]$ be of degree $\ge 2$.
Suppose $x$ lies in a finite extension $F$ of $K$.
Let $v\in M_\IQ$ and for a
ring homomorphism $\sigma\in\hom(F,\IC_v)$
we let $\sigma(f)$ be the polynomial obtained by applying $\sigma$ to
each coefficient of $f$. 
The \textit{canonical} or \textit{Call--Silverman}
height of $x$ (with respect to $f$) is
\begin{equation}
  \label{def:csheight}
  \hat h_f(x) = \frac{1}{[F:\IQ]} \sum_{v\in M_\IQ}
  \sum_{\sigma \in\hom(F,\IC_v)}
  \lambda_{\sigma(f),v}(\sigma(x)). 
\end{equation}
This value does not depend on the number field $F\supset K$ containing $x$. Let
$\overline K$ denote an algebraic closure of $K$, we obtain a
well-defined map $\hat h_f\colon \overline K\rightarrow [0,\infty)$.
For $x\in F$ we define
\begin{equation*}
  \lambda^{\mathrm{max}}_f(x) = \max\left\{ \lambda_{\sigma(f)}(\sigma(x))
  : \sigma\in \hom( F,\IC)\right\}. 
\end{equation*}

To formulate our results we require Dimitrov's notion of a hedgehog.
A \textit{quill} is a line segment $[0,1] z=\{tz : t\in [0,1]\}\subset \IC$
where $z\in\IC^\times = \IC\ssm\{0\}$. A \textit{hedgehog} with at
most $q$ quills is a finite union of
quills
\begin{equation*}
  \hh(z_1,\ldots,z_q) = \bigcup_{i=1}^q [0,1] z_i
\end{equation*}
where  $z_1,\ldots,z_q\in\IC^\times$; we also define
$\hh() = \{0\}$ to be the quillless hedgehog. A hedgehog is
a compact, path connected, topological tree.

We will also use Dubinin's Theorem~\cite{Dubinin} which  states
$$\mathrm{cap}(\hh(z_1,\ldots,z_q)) \le 4^{-1/q}
\max\{|z_1|,\ldots,|z_q|\}$$ if $q\ge 1$, where $\mathrm{cap}(\cK)$ is
the capacity, or equivalently the transfinite diameter,
of a compact subset $\cK\subset \IC$. We refer
to Section~\ref{sec:construction} for basic properties of the
capacity.

\begin{qhyp}
  Let $K$ be a number field and let $f\in K[T]$ have degree $\ge 2$.
  We say that $f$ satisfies the \emph{Quill Hypothesis}, if there exists
  $\sigma_0\in \hom(K,\IC)$ such that the post-critical set
  $\pco^+(\sigma_0(f))$ is contained in a hedgehog with at most $q\ge
  0$ quills with  
  \begin{equation}
    \label{eq:quillcondition}
    q \log \cf(\sigma_0(f)) < \log 4.
  \end{equation}
\end{qhyp}

This condition and the constant $4$ are artifacts of Dubinin's
Theorem. Our main results require $f$ to satisfy the Quill Hypothesis,
which should be unnecessary conjecturally.

Say $f=T^2-1$, then $\pco^{+}(f)$ is contained in the single quill
$[-1,0]$. We will see in Lemma~\ref{lem:iteratefabsbound} that
$\cf(T^2-1) \le (\sqrt 5+1)/2<4$. So $f=T^2-1$ satisfies the Quill
Hypothesis.

We are now ready to state our results.

\begin{thm}
  \label{thm:periodichgtlb}
  Let $p$ be a prime number, let $K$ be a number field, and let
  $f=T^p+c\in K[T]$. Suppose that $0$ is an $f$-periodic point
  and that $f$ satisfies the Quill Hypothesis.
  Then there exists $\kappa=\kappa(f)>0$ with the following properties. 
  \begin{enumerate}
  \item [(i)] Let $x\in\overline K$ be an algebraic integer and an $f$-wandering point, then
    $\lambda_f^{\mathrm{max}} (x) \ge \kappa/[K(x):K]$. 
  \item[(ii)] Let $x\in \overline K$ be an $f$-wandering point, then
    $\hat h_f (x) \ge \kappa/[K(x):K]^{2}$. 
  \end{enumerate}
\end{thm}

Part (i) can be seen as a dynamical Schinzel--Zassenhaus property for
$f$. Part (ii) of the theorem is conjectured to hold for all rational
functions $f\in K(T)$ of degree $\ge 2$ with $[K(x):K]^2$ replaced by
$[K(x):K]$. This is called the Dynamical Lehmer Conjecture, see
Conjecture 3.25~\cite{Silverman:gtm241} or Conjecture
16.2~\cite{BIJMST:trends} for a higher dimensional version. The
Dynamical Lehmer Conjecture is already open in the classical case
$f=T^2$.

In Theorem~\ref{thm:preperiodic2} below we prove a more general version of
Theorem~\ref{thm:periodichgtlb}.

For a general rational function, M.~Baker~\cite[Theorem
1.14]{MBaker:lowerdyngreen} obtained a lower bound for the canonical
height proportional to $1/[K(x):K]$ outside an exceptional set of
controlled cardinality. % Outside the exceptional set, this method
% leads to a lower bound for the canonical
% height that is exponential in $-[K(x):K]$.

Next we exhibit some cases where the Quill Hypothesis is satisfied.

\begin{lemma}
  \label{lem:quillhypothesis}
  Let $p$ be a prime number, let $K$ be a number field, and let
  $f=T^p+c\in K[T]$. Suppose that
  $f\not= T^2-2$. Then $f$ satisfies the Quill Hypothesis in both
  of the following cases. 
  \begin{enumerate}
  \item [(i)]  There exists a
    field embedding $\sigma_0\colon K\rightarrow \IR$ with
    $\pco^+(\sigma_0(f))$ bounded. 
  \item [(ii)]  There exists a
    field embedding $\sigma_0\colon K\rightarrow \IC$ with
    $\#\pco^+(\sigma_0(f)) \le 2p-2$.
 \end{enumerate}
\end{lemma}

Let $n\in\IN$ and let $G_n$ denote the monic polynomial whose roots
are precisely those $c\in\IC$ for which $0$ is $(T^2+c)$-periodic with
exact period $n$. Then $G_n\in\IZ[X]$ and work of
Lutzky~\cite{Lutzky:CountingHyperbolic} implies that $G_n$ has at
least one real root. In particular, the Quill hypothesis is met for
infinitely many quadratic polynomials $T^2+c$
with periodic critical point. It is a folklore conjecture that $G_n$ is
irreducible for all $n\in\IN$. Conditional on this conjecture any
$f=T^2+c$ with $c\not=-2$ for which $0$ is $f$-periodic satisfies the
Quill Hypothesis.

If $0$ is periodic of minimal period larger than 1, then $f = T^p + c$
is not conjugated to a Chebyshev polynomial or $T^p$. Using the lemma
above we find for each prime $p$ an $f$ not of this form that satisfies
the Quill Hypothesis.

The main technical results of this paper are in
Section~\ref{sec:proofs}. For example, we also cover some cases where
$0$ is an $f$-preperiodic point. We now draw several corollaries.

\begin{cor}
  \label{cor:irrfactorgrowth}
  Let $p$ be a prime number, let $K$ be a number field, and let
  $f=T^p+c\in K[T]$. Suppose that $0$ is an $f$-periodic point
  and that $f$ satisfies the Quill Hypothesis.
  Suppose $y\in K$.% is not an $f$-periodic point. 
  \begin{enumerate}
  \item[(i)] Let $y$ be an $f$-periodic point and $x\in\overline K$
    with $f^{(n)}(x)=y$ with $n\ge 0$ minimal, then $[K(x):K]\ge
    p^{n-\mathrm{per}(y)}$.    
  \item [(ii)]
    Let $y$ be an $f$-preperiodic point and not $f$-periodic, then $f^{(n)}-y\in K[T]$ is
    irreducible for all $n\in\IN$.
  \item [(iii)]
    Let $y$ be  an $f$-wandering point, then
    there exists $\kappa=\kappa(f,y)>0$
    such that for all  $n\in\IN$, each irreducible factor of
    $f^{(n)} - y \in K[T]$ has degree at least $\kappa p^n$.
  \end{enumerate}
\end{cor}

For example, if $f=T^2-1$ then $f^{(n)}-1$ is irreducible in $\IQ[T]$
for all $n\in\IN$. Stoll~\cite{Stoll:GaloisIterated} showed that
iterates of certain quadratic polynomials are irreducible over $\IQ$.
Part (iii) of Corollary~\ref{cor:irrfactorgrowth} was proved in more
general form by Jones and Levy \cite{JonesLevy:17} using different methods.

% Omitted
% Sookdeo~\cite{sookdeo:11} studied the backwards orbit under power
% maps. We conclude that Sookdeo's Conjecture 1.2~\textit{loc.cit.}
% holds for $f$ as in Corollary~\ref{cor:irrfactorgrowth} and $\beta =
% y\in K$; indeed it suffices to apply Sookdeo's Theorems 2.5 and 2.6 in
% \textit{loc.cit.}.

The method presented here is effective in that it can produce explicit
lower bounds for heights. In the next two corollaries we take a closer
look at the quadratic polynomial $f=T^2-1$.

\begin{cor}
  \label{cor:wanderingminus1}
  Let $f = T^2-1$.
  \begin{enumerate}
  \item [(i)]   Let $x\in\IQbar$ be  an algebraic integer and an
    $f$-wandering point, then
    \begin{equation*}
      \lambda^{\mathrm{max}}_{f}(x)\ge \frac{\log(4^{11/10}/(\sqrt 5+1))}{48}
      \frac{1}{[\IQ(x):\IQ]}> \frac{0.007}{[\IQ(x):\IQ]}.
    \end{equation*}
  \item[(ii)] Let $x\in\IQbar$ be an $f$-wandering point, then
    \begin{equation*}
      \hat h_f(x)\ge \frac{\log(4^{11/10}/(\sqrt 5+1))}{48}
      \frac{1}{[\IQ(x):\IQ]^2}> \frac{0.007}{[\IQ(x):\IQ]^2}.
    \end{equation*}
  \end{enumerate}
\end{cor}

\begin{cor}
  \label{cor:preperminus1}
  Let $f=T^2-1\in\IQ[T]$ and let
  $x\in\IQbar$ be an $f$-preperiodic point, then
  \begin{equation*}
    [\IQ(x):\IQ] \ge \frac 12 \max\{2^{\mathrm{preper}(x)},\mathrm{per}(x)\}. 
  \end{equation*}
\end{cor}

We now discuss the method of proof. It is an adaptation of Dimitrov's
proof of the Schinzel--Zassenhaus Conjecture~\cite{dimitrov:SZ}.

Suppose we are given $f$ as in Theorem~\ref{thm:periodichgtlb} and
an algebraic integer $x$ such that $\lambda_f^{\mathrm{max}}(x)$ is
sufficiently small. For simplicity say $f=T^p+c\in \IQ[T]$.

The $\IQ$-minimal polynomial $A\in\IQ[X]$ of $x$ factors as
$(X-x_1)\cdots(X-x_D)$ with $x_1,\ldots,x_D\in\IC$. For an integer
$k\ge 1$ we set $A_k = (X-f^{(k)}(x_1))\cdots(X-f^{(k)}(x_D))$. The
basic idea is to consider a $p$-th root $\phi$ of the rational
function $(A_l/A_k)^{p-1}=1+O(1/X)$ for certain integers $1\le k<l$.
\textit{A priori}, we may consider $\phi$ as a formal power series in
$\IQ[[1/X]]$. However, an appropriate choice of $k$ and $l$ will imply
$\phi\in\IZ[[1/X]]$. This formal power series is constructed in
Section~\ref{sec:powerseries} where we also analyze integrality
properties via congruence conditions. The power series $\phi$
respresents a holomorphic function on the complement of a large enough
disk. This step is similar as in
Dimitrov's proof.

% Suppose that the post-critical set of $f$ is finite.
Our next step
will be to construct a sufficiently nice domain $U\subset\IC$ on which
$\phi$ represents a holomorphic function.
%$A_l/A_k$ is regular and non-vanishing.
This construction is done in Section~\ref{sec:construction} and it is
the main new aspect of this paper. Let $\widehat\IC=\IC\cup\{\infty\}$
denote the extended complex plane. The construction of $U$ depends on
a parameter $n\in\IN$. It takes as an input a hedgehog $I_0$ that
contains the post-critical set of $f$ (or at least part of it).
The critical values of $f^{(n)}$ are contained in the preimage
$(f^{(n)})^{-1}(I_0)$; which one should think of as a finite
topological tree. 
% The
% complement $\widehat \IC\ssm (f^{(n)})^{-1}(I_0)$ is then simply
% connected. Moreover, we arrange that
Thus
$f^{(n)}$ has no critical points
on $\IC\ssm (f^{(n)})^{-1}(I_0)$.
Moreover, adding the point at infinity will make
$\IC\ssm (f^{(n)})^{-1}(I_0)$ simple connected. 
% Since $0\in I_0$, the iterate
% $f^{(n)}$ does not vanish on $\IC\ssm (f^{(n)})^{-1}(I_0)$. % By
% monodromy considerations we see that $z\mapsto f^{(n)}(z)^{1/p^n}$
% extends to a holomorphic map on $\IC\ssm (f^{(n)})^{-1}(I_0)$.

However, this complement is not yet the desired $U$. We need to ensure
that $U$ does not contain the $f^{(k)}(x_i)$ and $f^{(l)}(x_j)$ so
that $A_l/A_k$ has neither pole nor zero on $U$. This can be done by
augmenting $(f^{(n)})^{-1}(I_0)$ to a larger finite topological tree
while retaining the property that $U\cup\{\infty\}$ is simply
connected. By monodromy, the $p$-th root $\phi$ extends to a
holomorphic function on $U$.

Dubinin's Theorem is our tool to bound from above the capacity of a
suitable hedgehog. We will be able to relate  the capacity of  $\IC\ssm
U$ to the capacity of a hedgehog
using transformation properties of the capacity under certain holomorphic
mappings. 

%we will obtain an upper bound for $\mathrm{cap}(\IC\ssm U)$.

Finally, in Section~\ref{sec:proofs} we prove the results stated in
the introduction. As did Dimitrov, we use the P\'olya--Bertrandias
Theorem. It implies that $\phi$ is a rational function if
$\mathrm{cap}(\IC\ssm U)<1$. The Quill Hypothesis, the post-critical
finite hypothesis on $f$, and the fact that
$\lambda_f^{\mathrm{max}}(x)$ is small is used to ensure the capacity inequality. Here we
also choose the parameter $n$. If $\phi$, whose $p$-th power is
$(A_l/A_k)^{p-1}$, is a rational function then 
there must be cancellation and $A_k$ and $A_l$ have common roots.
For example, if $f^{(k)}(x_i) =
f^{(l)}(x_j)$ for some $i,j$ there is a collision of Galois conjugates
which entails that $x$ is $f$-preperiodic. 

\section*{Acknowledgments}

We thank Vesselin Dimitrov for suggesting to the second-named author
to investigate the dynamical Schinzel--Zassenhaus for $f=T^2-1$. The
authors thank Holly Krieger for comments on a draft of this text.
The first-named author has
received funding from the Swiss National Science Foundation project
n$^\circ$ 200020\_184623.

%%% Local Variables:
%%% TeX-master: "main"
%%% End:

\section{Constructing the Power Series}
\label{sec:powerseries}

In this section we construct an auxiliary power series. Our approach
largely follows the approach laid out in Section 2~\cite{dimitrov:SZ}
and in particular Proposition 2.2~\textit{loc.cit.} We adapt it to the
dynamical setting. Dimitrov's crucial observation is that
$\sqrt{1+4X}$ is a formal power series in $X$ with integral
coefficients.

Throughout this section let $R$ denote an integrally closed domain of
characteristic $0$. Let $p$ be a prime number.

For $D\in\IN$ and $j\in \{0,\ldots,D\}$ let $e_j
\in\IZ[X_1,\ldots,X_D]$ denote the elementary symmetric polynomial of
degree $j$. We set
\begin{equation*}
%  \label{def:Psi}
  \Psi_{p,j} = \frac 1p \left( e_j(X_1,\ldots,X_D)^p - e_j(X_1^p,\ldots,X_D^p)\right).
\end{equation*}
So $\Psi_{p,j}\in \IZ[X_1,\ldots,X_D]$ by Fermat's Little Theorem
and 
\begin{equation}
  \label{eq:Psi2}
  e_j(X_1,\ldots,X_D)^p = e_j(X_1^p,\ldots,X_D^p) + p \Psi_{p,j}. 
\end{equation}

If $R$ is subring of a ring $S$ and $a,b\in S$ we write $a\equiv b
\imod{p}$ or $a\equiv b\imod {pS}$ to signify $a-b\in pS$.

\begin{lemma}
  \label{lem:powersumlemma}
  For all $k\in\IN$ and all $j\in
  \{0,\ldots,D\}$ 
  we have
  \begin{equation}
    \label{eq:powersumlemma}
    e_j(X_1,\ldots,X_D)^{p^k} \equiv
    e_j(X_1^{p^k},\ldots,X_D^{p^k})
    + p \Psi_{p,j}^{p^{k-1}}
    \imod {p^2}. 
  \end{equation}
\end{lemma}
\begin{proof}
  By (\ref{eq:Psi2}) our claim (\ref{eq:powersumlemma}) holds for
  $k=1$.

  We assume that (\ref{eq:powersumlemma}) holds for $k\ge 1$
  and will deduce it
  for $k+1$. So
  \begin{equation}    
    \label{eq:symm1}
    \begin{aligned}
    e_j(X_1,\ldots,X_D)^{p^{k+1}}
    &\equiv \left(
      e_j(X_1^{p^k},\ldots,X_D^{p^k})
      +p\Psi_{p,j}^{p^{k-1}}\right)^p \imod {p^2}
     \\
     &\equiv
     e_j(X_1^{p^k},\ldots,X_D^{p^k})^p \imod{p^2}
   \end{aligned}
 \end{equation}
 as $(x+py)^p \equiv x^p \imod {p^2}$.  

  We evaluate $\Psi_{p,j}$ at $(X_1^{p^k},\ldots,X_D^{p^k})$
  and use (\ref{eq:Psi2}) to  find 
  \begin{equation}
    \label{eq:symm2}
    \begin{aligned}
      e_j(X_1^{p^k},\ldots,X_D^{p^k})^p
      = e_j(X_1^{p^{k+1}},\ldots,X_D^{p^{k+1}}) +
      p \Psi_{p,j}(X_1^{p^k},\ldots,X_D^{p^k}).
    \end{aligned}
  \end{equation}
  
  Finally,  
  $\Psi_{p,j}(X_1^{p^k},\ldots,X_D^{p^k}) \in \Psi_{p,j}^{p^k}
  + p\IZ[X_1,\ldots,X_D]$ and hence
  $e_j(X_1^{p^k},\ldots,X_D^{p^k})^p
      \equiv e_j(X_1^{p^{k+1}},\ldots,X_D^{p^{k+1}}) +
      p \Psi_{p,j}^{p^k} \imod{p^2}$ from  (\ref{eq:symm2}).  
  Thus (\ref{eq:powersumlemma}) for $k+1$
  follows from
  (\ref{eq:symm1}).
\end{proof}

\begin{lemma}
  \label{lem:powerpoly}
  Let $S$ be a ring of which $R$ is a subring and
%  Let $S$ be a ring extension of $S$ and 
%  Let $S$ be a ring, let $R$ be a subring of $S$ and
  suppose $x_1,\ldots,x_D\in S$ such that
  $(X-x_1)\cdots(X-x_D)\in R[X]$. 
  Let $p$ be a prime number. 
  \begin{enumerate}
  \item [(i)]
    The value of a
    symmetric polynomial in $R[X_1,\ldots,X_D]$ evaluated at
    $(x_1,\ldots,x_D)$ lies in $R$. 
  \item[(ii)]   
    Let $k,l\in\IN$ with
    $a^{p^{k-1}}\equiv a^{p^{l-1}}\imod{pR}$ for all $a\in R$.  
    For all $j\in \{0,\ldots,D\}$ we have
    \begin{equation*}
      e_j(x_1^{p^k},\ldots,x_D^{p^k}) \equiv
      e_j(x_1^{p^l},\ldots,x_D^{p^l}) 
      \imod{p^2R}.
    \end{equation*}
  \end{enumerate}
\end{lemma}
\begin{proof}
  An elementary symmetric polynomial in $R[X_1,\ldots,X_D]$ lies in $R[e_1,\ldots,e_D]$.
  We conclude (i) as
  $e_1(x_1,\ldots,x_D),\ldots,e_D(x_1,\ldots,x_D)\in R$ by hypothesis.

  For (ii) 
  observe that $\Psi_{p,j}$ is  symmetric
  so $\Psi_{p,j}(x_1,\ldots,x_D)\in R$ by part (i).
  % and our hypothesis
  % on $x_1,\ldots,x_D$.
  For the same reason $e_j(x_1^{p^k},\ldots,x_D^{p^k})\in R$
  and $e_j(x_1^{p^l},\ldots,x_D^{p^l})\in R$.

  By the hypothesis in (ii) we conclude
  $\Psi_{p,j}(x_1,\ldots,x_D)^{p^{k-1}}\equiv
  \Psi_{p,j}(x_1,\ldots,x_D)^{p^{l-1}} \imod{p R}$.  
  Lemma~\ref{lem:powersumlemma} implies
  \begin{equation}
    \label{eq:powersummodpsqr}
    e_j(x_1,\ldots,x_D)^{p^k} - e_j(x_1^{p^k},\ldots,x_D^{p^k})
    \equiv
    e_j(x_1,\ldots,x_D)^{p^l} - e_j(x_1^{p^l},\ldots,x_D^{p^l})
    \imod{p^2 R}.
  \end{equation}

  Raising 
  $a^{p^{k-1}}\equiv a^{p^{l-1}}\imod{p R}$ to the $p$-th power gives
  $a^{p^k} \equiv a^{p^l} \imod {p^2 R}$ for all $a\in R$. 
  In particular, 
  $e_j(x_1,\ldots,x_D)^{p^k}$ and 
  $e_j(x_1,\ldots,x_D)^{p^l}$ are equivalent modulo $p^2$.
  Part (ii) now follows from (\ref{eq:powersummodpsqr}).\qedhere
\end{proof}

%Now let us consider a polynomial $f = T^p+c$ with $c\in R$.

\begin{lemma}
  \label{lem:congpsqrlemma}
  Let 
  $f = T^p+c\in R[T]$.
  We have
  \begin{equation*}
    f^{(k)} \equiv \left(T^{p^{k-1}}+f^{(k-1)}(0) \right)^p+c \imod {p^2R[T]}  
  \end{equation*}
  for all $k\in \IN$.
 %  Let $a_1=0$ and $c_1=c$.
 %   We define recursively $a_{k+1} = 2c_k$ and
 %  $c_{k+1} = c_k^2 + c$ for all $k\ge 2$.  Then
 %  \begin{equation*}
 %   f^{(k)} \equiv T^{2^{k}} + a_k T^{2^{k-1}} + c_k \imod {4R}  
 % \end{equation*}
 % for all $k\in \IN$.
\end{lemma}
\begin{proof}
  As $f^{(0)}(0)=0$ the claim holds for $k=1$.
  Say $k\ge 2$. By induction on $k$ we have
  $f^{(k-1)} \equiv T^{p^{k-1}}+p g + f^{(k-2)}(0)^p+c \imod
  {p^2R[T]}$ for some $g\in R[T]$. So
  \begin{alignat*}1
    f^{(k)} &\equiv \left(T^{p^{k-1}}+p g + f^{(k-2)}(0)^p+c\right)^p + c \imod
    {p^2R[T]}\\
    &\equiv \left(T^{p^{k-1}}+ f^{(k-2)}(0)^p+c\right)^p +c \imod
    {p^2R[T]}
  \end{alignat*}
  and the lemma follows from $f^{(k-1)}(0) = f^{(k-2)}(0)^p+c$. 
\end{proof}

% For all $k\in\IN$ we define
% \begin{equation}
%   \label{def:gk}
%   g_k = (T+f^{(k-1)}(0))^p+c \in R[T]. 
% \end{equation}
% Then Lemma~\ref{lem:congpsqrlemma} states
% \begin{equation}
%   \label{eq:fkgkprop}
%   f^{(k)} \equiv  g_k(T^{p^{k-1}}) \imod {p^2R[T]}.
% \end{equation}

Let $f\in R[T]$ and $A=\prod_{j=1}^D (X-x_j)\in R[X]$ where
$x_1,\ldots,x_D$ are in a splitting field of $A$. For all integers $k\ge 0$
we set
$$A_k = \prod_{j=1}^D (X-f^{(k)}(x_j)).$$
Then $A_k\in R[X]$ by Lemma~\ref{lem:powerpoly}(i).

\begin{lemma}
  %\label{lem:conglemma1}
  \label{lem:AkAlcong}
  % Let $R$ be an integrally closed domain
  % and
  Let $ f= T^p+c\in R[T]$. 
  Suppose that $k,l\in\IN$ satisfy $k\le l$ and $f^{(k-1)}(0) =
  f^{(l-1)}(0)$.
  Set
  \begin{equation}
    \label{def:delta}
    \delta = \left\{
      \begin{array}{ll}
        0 &: \text{if }k=1,\\
        1 &:\text{if }k\ge 2.             
      \end{array}
    \right.
  \end{equation}
  We assume that
  \begin{equation}
    \label{eq:congprops}
    a^{p^{k-\delta-1}}\equiv a^{p^{l-\delta-1}} \imod{p}
    \quad    \text{for all }a\in R.
  \end{equation}
  Then $A_k\equiv A_l \imod{p^2 R[X]}$ for all monic
  $A\in R[X]$. 
\end{lemma}
\begin{proof}
  If $k=1$, then $0=f^{(k-1)}(0)=f^{(l-1)}(0)$. In this case,
  Lemma~\ref{lem:congpsqrlemma}  implies
  $f^{(l)} \equiv g(T^{p^l})\imod {p^2R[T]}$ with $g = T+c$.
  Observe that $f^{(k)} = f = g(T^p) = g(T^{p^k})$.

  If $k\ge 2$, Lemma~\ref{lem:congpsqrlemma} implies
  $f^{(k)}\equiv g(T^{p^{k-1}}) \imod{p^2R[T]}$ and
  $f^{(l)}\equiv g(T^{p^{l-1}}) \imod{p^2R[T]}$ where
  $g = (T+f^{(k-1)}(0))^p+c$.

  We can summarize both cases by $f^{(k)}\equiv g(T^{p^{k-\delta}}) \imod{p^2}$
  and $f^{(l)}\equiv g(T^{p^{l-\delta}}) \imod{p^2}$; we sometimes drop the
  reference to base ring in $\imod {p^2}$ below.
  
  Let $j\in\{0,\ldots,D\}$, the polynomial 
  $e_j(g(X_1),\ldots,g(X_D))$ is symmetric.  So
  \begin{equation*}
    e_j(g(X_1),\ldots,g(X_D))= {P_{j}}(e_1,\ldots,e_D)
  \end{equation*}
  for some ${P_{j}} \in R[X_1,\ldots,X_D]$. % From the hypothesis $g_k\equiv g_l
  Next we replace $X_i$ by $X_i^{p^{k-\delta}}$ to get
  \begin{equation*}
    e_j(g(X_1^{p^{k-\delta}}),\ldots,g(X_D^{p^{k-\delta}}))=     P_{j}(e_1(X_1^{p^{k-\delta}},\ldots,X_D^{p^{k-\delta}}),\ldots,e_D(X_1^{p^{k-\delta}},\ldots,X_D^{p^{k-\delta}})).
  \end{equation*}
  Hence
  \begin{equation}
    \label{eq:sjPjk}
    e_j(f^{(k)}(X_1),\ldots,f^{(k)}(X_1))\equiv
    P_{j}(e_1(X_1^{p^{k-\delta}},\ldots,X_D^{p^{k-\delta}}),\ldots,e_D(X_1^{p^{k-\delta}},\ldots,X_D^{p^{k-\delta}}))\imod{p^2}.
  \end{equation}

  The same argument for $l$ yields
  \begin{equation}
    \label{eq:sjPjl}
    e_j(f^{(l)}(X_1),\ldots,f^{(l)}(X_D)) \equiv
    P_{j}(e_1(X_1^{p^{l-\delta}},\ldots,X_D^{p^{l-\delta}}),\ldots,e_D(X_1^{p^{l-\delta}},\ldots,X_D^{p^{l-\delta}}))
    \imod {p^2}.
  \end{equation}
  
  We factor $A = (X-x_1)\cdots (X-x_D)$ with $x_1,\ldots,x_D$ inside a
  fixed splitting field of $A$. We
  specialize $(x_1,\ldots,x_D)$ in
  (\ref{eq:sjPjk}) to get
  \begin{equation*}
    e_j\left(f^{(k)}(x_1),\ldots,f^{(k)}(x_D)\right)\in  
    P_{j}\left(e_1(x_1^{p^{k-\delta}},\ldots,x_D^{p^{k-\delta}}),\ldots,e_D(x_1^{p^{k-\delta}},\ldots,x_D^{p^{k-\delta}})\right)+p^2R[x_1,\ldots,x_D];
  \end{equation*}
  both elements lie in $R$ by
  Lemma~\ref{lem:powerpoly}(i). As any
  element in $R[x_1,\ldots,x_D]$ is integral over $R$ and since $R$ is
  integrally closed we conclude 
  \begin{equation*}
    e_j(f^{(k)}(x_1),\ldots,f^{(k)}(x_D))
    \equiv
    P_{j}\left(e_1(x_1^{p^{k-\delta}},\ldots,x_D^{p^{k-\delta}}),\ldots,e_D(x_1^{p^{k-\delta}},\ldots,x_D^{p^{k-\delta}})\right)\imod{p^2
    R}.
  \end{equation*}
  The same statement holds with $k$ replaced by $l$ by using (\ref{eq:sjPjl}).

  By hypothesis (\ref{eq:congprops})
  and by Lemma~\ref{lem:powerpoly}(ii) applied to $k-\delta,l-\delta$ we have
  $$ e_j(x_1^{p^{k-\delta}},\ldots,x_D^{p^{k-\delta}}) \equiv
  e_j(x_1^{p^{l-\delta}},\ldots,x_D^{p^{l-\delta}}) \imod {p^2R}$$
  and hence
  $$    e_j(f^{(k)}(x_1),\ldots,f^{(k)}(x_D))
  \equiv     e_j(f^{(l)}(x_1),\ldots,f^{(l)}(x_D))\imod{p^2R}.$$
  This holds for all $j$ and so 
  $A_k \equiv A_l \imod {p^2}$, as desired.
\end{proof}

Let
\begin{equation}
  \label{def:Phip}
  \Phi_p = \sum_{k=0}^\infty {1/p \choose k} U^k \in \IQ[[U]],
\end{equation}
it satisfies $\Phi_p^p = 1+U$. If $k\ge 0$ is an integer, then
\begin{equation*}
  {1/p\choose k} p^k = \frac{1}{k!} \frac 1p\left(\frac 1p
    -1\right)\cdots\left(\frac 1p -k+1\right) p^k \in \frac{1}{k!}\IZ.
\end{equation*}
The exponent of the prime  $p$ in $k!$ equals the finite sum
$\sum_{e=1}^{\infty}[k/p^e]$ which is at most
$k\sum_{e=1}^\infty p^{-e}= k/(p-1)\le k$. 
So $p$ does not appear in the
denominator of 
${1/p\choose k} p^{2k}\in\IQ$ for all $k\ge 0$.
Moreover, if $\ell$ is a prime distinct from $p$, then $\ell$ does not
appear in the denominator of ${1/p\choose k}$. 
Therefore, 
\begin{equation}
  \label{eq:psqrintegral}
  \Phi_p(p^2U) \in \IZ[[U]]. 
\end{equation}

Let $R[[1/X]]$ be the ring of formal  power series in $X$.
We write $\mathrm{ord}_\infty$ for the natural valuation on
$R[[1/X]]$, \textit{i.e.}, $\mathrm{ord}_\infty(1/X) = 1$. A
non-constant polynomial in $R[X]$ does not lie in $R[[1/X]]$, rather
it lies in the ring of formal Laurent series $R((1/X))$ to which we
extend $\mathrm{ord}_\infty$ in the unique fashion. So we will
consider $R[X]\subset R((1/X))$. For example, any monic polynomial in
$R[X]$ is a unit in $R((1/X))$.

\begin{prop}
  \label{prop:powerseries}
  %Let $R$ be an integrally closed domain,
  Let $p$ be a prime number and let $f=T^p+c\in R[T]$.
  Suppose that $k,l\in\IN$ satisfy $k\le l$ and $f^{(k-1)}(0) =
  f^{(l-1)}(0)$.
  Let $\delta\in \{0,1\}$ be as in Lemma~\ref{lem:AkAlcong} and suppose
  $a^{p^{k-\delta-1}}\equiv
  a^{p^{l-\delta-1}}\imod{p}$ for all $a\in R$. If $A\in R[T]$ is monic there
  exists $B\in R[X]$ such that
  $\deg B \le \deg A_k^p - 1$, 
  %$\mathrm{ord}_\infty(B/A_k^p) \ge 1$,
  $\Phi_p(p^2 B/A_k^p) \in R[[1/X]]$, and $\Phi_p(p^2 B/A_k^p)^p
  =(A_l/A_k)^{p-1}$.
\end{prop}
\begin{proof}
  Lemma~\ref{lem:AkAlcong} implies $A_k A_l^{p-1} = A_k^p + p^2 B$ for
  some $ B\in R[X]$. As $A_k$ and $A_l$ are both monic of degree $D$,
  leading terms cancel and we find $\deg B \le pD-1$. Then,
  \begin{equation*}
    A_k A_l^{p-1} = \left(1+p^2C\right)A_k^p.
  \end{equation*}
  with $C={B}/{A_k^p}\in R((1/X))$. So $\mathrm{ord}_\infty(C) =-\deg B + p D \ge 1$. Therefore, $C$ 
  is a formal power series in $1/X$ and even lies in the ideal $(1/X)R[[1/X]]$.

  Recall that $\Phi_p(p^2 U) \in \IZ[[U]]$ by (\ref{eq:psqrintegral}).
  Therefore, $\Phi_p(p^2C ) \in R[[1/X]]$ and
  \begin{equation*}
    A_kA_l^{p-1}=\left(\Phi_p(p^2C) A_k\right)^p
  \end{equation*}
  by the functional equation $\Phi_p^p=1+U$; this is an identity in
  $R((1/X))$. We conclude the proof by dividing by $A_k^{p-1}$.
\end{proof}  

%%% Local Variables:
%%% TeX-master: "main"
%%% End:

\section{Constructing Simply Connected Domains}
\label{sec:construction}
  
We begin with a topological lemma which is most likely
well-known. We provide a proof as we could not find a proper reference.

Below $\ecp$ denotes the extended complex plane $\IC\cup\{\infty\}$, 
\textit{i.e.}, the Riemann sphere.

\begin{lemma}
  \label{lem:topology}
  Let $f\in\IC[T]\ssm\IC$. Let $B\subset \IC$ be compact such that
  $\IC\ssm B$ is connected with cyclic fundamental group. If $B$
  contains all critical values of $f$, then $\IC\ssm f^{-1}(B)$ is
  connected and $\ecp\ssm f^{-1}(B)$ is a simply connected.
\end{lemma}
\begin{proof}
  Let us denote the restriction $f|_{f^{-1}(\IC\ssm B)} \colon
  f^{-1}(\IC\ssm B)=\IC \ssm f^{-1}(B)\rightarrow \IC\ssm B$ by $\psi$.
  Then $\psi$ is a surjective continuous map. As a non-constant polynomial
  induces a proper map we see that $\psi$ is proper and thus closed. By hypothesis,
  $f^{-1}(\IC\ssm B)$ does not contain any critical point of $f$. So $\psi$
  is a local homeomorphism and thus open. As $\psi$ is closed and open, the image
  of a connected component of $\IC\ssm f^{-1}(B)$ under $\psi$ is all of
  $\IC\ssm B$. As $\psi$ is the restriction of a polynomial, we conclude
  that all connected components of $\IC\ssm f^{-1}(B)$ are unbounded. 
  The preimage $f^{-1}(B)$ is compact and so $\IC \ssm f^{-1}(B)$ has
  a only one unbounded connected component. Therefore, 
  $\IC\ssm f^{-1}(B)$ is connected and so is $\ecp\ssm f^{-1}(B)$.

  We fix a base point $z_0\in \IC\ssm f^{-1}(B)$.
  The proper, surjective, local homeomorphism $\psi$ is
  a topological covering by
  Lemma~2~\cite{Ho:propermaps}.
  Then $\pi_1(\IC \ssm f^{-1}(B),z_0)$ is isomorphic to a
  subgroup of $\pi_1(\IC \ssm B,p(z_0))$ of finite index and thus
  cyclic. 

  Let $U_\infty$ be the complement in $\IC$ of a closed disk centered
  at $0$
  containing the compact set $f^{-1}(B)$. We suppose $z_0\in
  U_\infty$. Adding the point at infinity gives us a simply connected
  domain $U_\infty\cup\{\infty\}\subset \ecp$ and we have $\ecp \ssm
  f^{-1}(B) = (U_\infty \cup\{\infty\})\cup (\IC\ssm f^{-1}(B))$. The
  Theorem of Seifert and van Kampen tells us that the inclusion induces
  a short exact sequence
  \begin{equation}
    \label{eq:shortexactseq}
    1\rightarrow \pi_1(U_\infty,z_0)\rightarrow \pi_1(\IC\ssm
    f^{-1}(B),z_0) \rightarrow \pi_1(\ecp\ssm
    f^{-1}(B),z_0)\rightarrow 1.
  \end{equation}

  We may assume $B\not=\emptyset$.   Consider the loop $t\mapsto z_0 e^{2\pi i t}$, its image lies in
  $U_\infty$. 
  Fix any $w\in f^{-1}(B)$. It is well-known that such loop
  represents a generator of $\pi_1(\IC \ssm \{w\},z_0)\cong\IZ$ and
  that 
  the natural homomorphism $\pi_1(U_\infty,z_0)\rightarrow
  \pi_1(\IC\ssm f^{-1}(B),z_0) \rightarrow \pi_1(\IC\ssm\{w\},z_0)$ is
  an isomorphism. So also the middle group is isomorphic to $\IZ$ and the
  first arrow is a group isomorphism. 
  Therefore, $\ecp\ssm f^{-1}(B)$ is simply connected by
  (\ref{eq:shortexactseq}). 
\end{proof}

We recall the definition of the transfinite diameter of a
non-empty compact subset $\cK$ of $\IC$; see Chapter
5.5~\cite{Ransford} for  details. For an integer $n\ge 2$ we
define
\begin{equation*}
  d_n(\cK) = \sup \left\{ \prod_{1\le i<j\le n} |z_i-z_j|^{2/(n(n-1))} :
    z_1,\ldots,z_n\in\cK\right\}. 
\end{equation*}
The sequence $\bigl(d_n(\cK)\bigr)_{n\ge 2}$ is non-increasing and the
transfinite diameter of $\cK$ is its limit. The transfinite diameter
is known to equal the capacity $\mathrm{cap}(\cK)$ of $\cK$, by the
Fekete--Szeg\"o Theorem, see Theorem~5.5.2~\cite{Ransford}.
%We denote the common value by $\mathrm{cap}(\cK)$.
It is convenient to define $\mathrm{cap}(\emptyset)=0$.

For $f\in\IC[T]$ and for $n\in\IN$ we introduce the truncated
post-critical set
\begin{equation*}
  \pco^+_n(f) = \bigl\{f^{(m)}(x) : m\in\{1,\ldots,n\},x\in\IC, \text{ and
  }f'(x)=0\bigr\}
  \subset \pco^+(f).
\end{equation*}
In the next proposition we will use Dubinin's Theorem. 
% We do not require that $f$ is post-critically finite but rather that
% $\pco_{n}^+(f)$ is contained in a hedgehog.

\begin{prop}
  \label{prop:constructK2}
  Let $f\in\IC[T]$ be monic of degree $d\ge 2$ such that $f'(0)=0$.
  Let $n\in\IN$ such that 
  there exist $q\ge 0$ and   $\gamma_1,\ldots,\gamma_q\in\IC^\times$
  with 
  $\pco_{n}^+(f)\subset\hh(\gamma_1,\ldots,\gamma_q)$.
  Let $m\ge 1$ and let $\alpha_1,\ldots,\alpha_m \in\IC$.
  There is a simply connected domain $U\subset\ecp$
  with the following properties.
  \begin{enumerate}
  \item[(i)] We have   $0,\alpha_1,\ldots,\alpha_m\not\in U$ and
    $\ecp\ssm U$ is a compact subset of $\IC$.
  \item[(ii)] We have
    \begin{equation*}
%      \label{eq:capKbound1}
      \mathrm{cap}(\ecp\ssm U)\le
      4^{-1/(qd^n+m)}
      \max\{|\gamma_1|,\ldots,|\gamma_q|,|f^{(n)}(\alpha_1)|,\ldots,|f^{(n)}(\alpha_m)|\}^{1/d^n}.
      % \max\{|\gamma_1|^{1/d^n},\ldots,|\gamma_q|^{1/d^n},e^{\cf(f)/d^n +
      %   \lambda}\}
    \end{equation*}
    % where $\lambda =
    % \max\{\lambda_f(\alpha_1),\ldots,\lambda_f(\alpha_m)\}$
    % and where $\cf(f)\ge 0$ satisfies (\ref{eq:cff}).
  \end{enumerate}
\end{prop}

To prepare for the proof of Proposition~\ref{prop:constructK2} let $f\in\IC[T]$
denote a monic polynomial of degree $d\ge 2$ with $f'(0)=0$. Let
$n\in\IN$ and say $\gamma_1,\ldots,\gamma_q\in\IC^\times$ such that
$I_0=\hh(\gamma_1,\ldots,\gamma_q)$ contains $\pco_{n}^+(f)$.

We apply Lemma~\ref{lem:topology} to $f^{(n)}$ and $B=I_0$; the
fundamental group of $\IC\ssm I_0$ is cyclic and generated by a large
enough loop. We set
$$H_f^{(n)} = (f^{(n)})^{-1}(I_0).$$
By Lemma~\ref{lem:topology} the complement $\IC\ssm H_f^{(n)}$ is
connected and $\ecp \ssm H_f^{(n)}$ is simply connected. Observe that
$0\in H_f^{(n)}$ because $0$ is a critical point of $f$.

We will show that a branch of $z\mapsto f^{(n)}(z)^{1/d^n}$
admits holomorphic continuation on $\IC\ssm H_f^{(n)}$.
Moreover, this continuation will be a biholomorphic map between
$\IC\ssm H_f^{(n)}$ and the complement of a hedgehog. 

If $z\in \IC$ with $f^{(n)}(z)=0$, then $z\in H_f^{(n)}$ as $0\in
I_0$. Therefore, $f^{(n)}$ has no roots in $\IC\ssm H_f^{(n)}$. The
association $w\mapsto f^{(n)}(w^{-1})^{-1}$ is holomorphic on $\IC
\ssm \tilde H$ with $\tilde H = \{w \in \IC^\times : w^{-1} \in
H_f^{(n)}\}$ closed; the singularity at $0\in \IC\ssm \tilde H$ is
removable and is a zero of order $d^n$. It follows that $\tilde f
\colon w\mapsto w^{-d^n} f^{(n)}(w^{-1})^{-1}$ is holomorphic on $\IC
\ssm \tilde H$ without zeros. Moreover, $\tilde f(0)=1$ as $f$ is
monic.

The complements $\ecp \ssm H_f^{(n)}$ and $\IC \ssm \tilde H$ are
homeomorphic via the M\"obius transformation $z\mapsto z^{-1}$.
In particular, the latter is simply
connected. So
there exists a holomorphic function $\tilde g_n \colon
\IC\ssm \tilde H \rightarrow \IC$ with $\tilde g_n(w)^{d^n}=\tilde
f(w)$ for all $w\in\IC\ssm\tilde H$ and with $\tilde g_n(0)=1$. We
resubstitute $z=1/w$ and define $g_n(z) = z\tilde g_n(z^{-1})^{-1}$ to
find
\begin{equation}
  \label{eq:funceq}
  g_n(z)^{d^n} = f^{(n)}(z)
\end{equation}
for all $z\in \IC \ssm H_f^{(n)}$. Then $g_n$ is holomorphic and
$g_n(z) = z + O(1)$ as $z\rightarrow\infty$.

% Let $z\in \IC \ssm H_f^{(n)}$, we claim that $g_n(z)\not\in I_n$.
% Indeed, otherwise $f^{(n)}(z)=g_n(z)^{d^n}\in I_0$.
Let $z\in \IC \ssm H_f^{(n)}$. 
Then $g_n(z) \in\IC\ssm I_n$ with  
$I_n = \{z\in \IC : z^{d^n} \in I_0\}$ by (\ref{eq:funceq}).
Now 
\begin{equation}
  \label{def:In}
  I_n=\hh\bigl(e^{2\pi \sqrt{-1}j/d^n}w_i\bigr)_{\substack{1\le i\le q \\ 1\le j \le
      d^n}}
\end{equation}
is a hedgehog with at most $q^dn$ quills
where the $w_i$ satisfy $w_i^{d^n}=\gamma_i$. 
% So $I_n$ is a hedgehog with at most $qd^n$ quills, each of length at most
%  $\max\{|\gamma_1|,\ldots,|\gamma_q|\}^{1/d^n}$. 

% and hence $z\in ({f^{(n)}})^{-1}(I_0)= H_f^{(n)}$, a contradiction. 
% So $g_n(\IC\ssm H_f^{(n)}) \subset \IC\ssm I_n$.

\begin{lemma}
  The holomorphic function $g_n$ induces a bijection
  $\IC\ssm H_f^{(n)}\rightarrow \IC\ssm I_n$.
\end{lemma}
\begin{proof}
  Suppose $z\in\IC\ssm H_f^{(n)}$ is a critical point of $g_n$, then
  deriving (\ref{eq:funceq}) gives $d^n g_n(z)^{d^n-1}g_n'(z) =
  ({f^{(n)}})'(z)$ and hence $(f^{(n)})'(z)=0$. The chain rule yields
  $f'(f^{(n-1)}(z))\cdots f'(z)=0$, so $f'$ vanishes at $f^{(j)}(z)$ for
  some $j\in \{0,\ldots,n-1\}$. Hence $f^{(n)}(z) \in \pco_{n}^+(f)$.
  But $\pco_n^+(f) \subset I_0$ by our standing hypothesis and so $z\in
  H_f^{(n)}$. This is a contradiction. We conclude that $g_n$ has no
  critical points on its domain $\IC\ssm H_f^{(n)}$ and that it is a
  local homeomorphism.
  
  Let us verify that the map
  $\IC\ssm H_f^{(n)}\rightarrow \IC\ssm I_n$ is proper. Indeed, let
  $\cK$ be compact with $\cK\subset \IC\ssm I_n$. We need to verify that
  $g_n^{-1}(\cK)$ is compact. To this end let $(x_k)_{k\in\IN}$ be a
  sequence in $g_n^{-1}(\cK)$.
  If the sequence is unbounded, then so is $(f^{(n)}(x_k))_{k\in\IN}$.
  This is impossible as $f^{(n)}(x_k) = g_n(x_k)^{d^n}$ lies in the
  image of $\cK$ in $\IC$ under the $d^n$-th power map, itself a bounded
  set.
  So we may replace $(x_k)_{k\in\IN}$ by a convergent subsequence
  with limit $x\in\IC$.
  To verify our claim it suffices to check $x\in g_n^{-1}(\cK)$.
  If $x\in \IC\ssm H_f^{(n)}$, then
  $g_n(x)=\lim_{k\rightarrow \infty} g_n(x_k) \in \cK$ and we are
  done.
  Let us assume $x\in H_f^{(n)}$. 
  As $\cK$ is compact we
  may pass to a subsequence for which $(g_n(x_k))_{k\in\IN}$
  converges to $y\in \cK$. 
  Now $y^{d^n}=\lim_{k\rightarrow\infty} f^{(n)}(x_k) = f^{(n)}(x)\in
  I_0$ using that $f^{(n)}$ is a polynomial.
  Hence $y\in I_n$, but this contradicts $\cK\subset \IC\ssm I_n$.
 
  The continuous map $\IC\ssm H_f^{(n)}\rightarrow \IC\ssm I_n$ is a
  proper local homeomorphism. So it is open and closed, hence surjective
  as both $\IC\ssm H_f^{(n)}$ and $\IC\ssm I_n$ are connected. We find
  that $\IC\ssm H_f^{(n)}\rightarrow \IC\ssm I_n$ is a topological
  covering using again Lemma~2~\cite{Ho:propermaps}.

  Recall that $I_0$ contains all critical values of $z\mapsto z^{d^n}$ and
  $z\mapsto f^{(n)}(z)$. The fiber of either map above a point of
  $\IC\ssm I_0$ contains $d^n = \deg f^{(n)}$ elements.
  So (\ref{eq:funceq}) implies that $g_n$ is injective.
\end{proof}

The inverse of a holomorphic bijection is again holomorphic. So the inverse
\begin{equation*}
  h_n \colon \IC\ssm I_n\rightarrow \IC\ssm H_f^{(n)}
\end{equation*}
of $g_n$ is also holomorphic;
the  function $h_n$ is a branch of $z\mapsto (f^{(n)})^{-1}(z^{d^n})$.
We have
\begin{equation*}
%  \label{eq:hsinfinity}
  h_n(z) = z + O(1)\text{ for }
  z\rightarrow\infty
\end{equation*}
by the same property of $g_n$.  
Thus $h_n$ extends to a homeomorphism
  $\ecp\ssm I_n\rightarrow \ecp\ssm H_f^{(n)}$.

\begin{proof}[Proof of Proposition~\ref{prop:constructK2}]
  If $\alpha_i\in
  \IC\ssm H_f^{(n)}$, then the value $g_n(\alpha_i)$ is well-defined.
  We define a further hedgehog
  \begin{equation}
    \label{def:Iprimen}
    I'_n = I_n \cup \hh\bigl( g_n(\alpha_i)\bigr)_{\substack{1\le i\le
        m \\ \alpha_i\not\in
      H_f^{(n)}}}
  \end{equation}
  with at most $qd^n+m$ quills, see (\ref{def:In}). 
  
  The hedgehog complement $\ecp\ssm I'_n$
  is connected and
  therefore so is
  $$U = h_n(\ecp\ssm I'_n).$$
  The image $U$ lies open in $\ecp\ssm H_f^{(n)}$ and thus in $\ecp$.
  Because $\infty\in U$ we find that $\ecp\ssm U=\IC\ssm U$ is
  compact; this yields part of the claim of (i). 
  The complement of $U$ is obtained by augmenting $H_f^{(n)}$, more precisely
  \begin{equation*}
    \IC \ssm U =  H_f^{(n)} \cup h_n\left(I'_n\ssm I_n\right). 
  \end{equation*}
  
  Recall $0 \in H_f^{(n)}$, so $0\not\in U$. Suppose $\alpha_i\in U$
  for some $i$. So $\alpha_i\not\in H_f^{(n)}$ and thus $g_n(\alpha_i) \in
  I'_n$ by construction. But $\alpha_i$ also lies in the image of $h_n$,
  \textit{i.e.}, $\alpha_i=h_n(\beta)$ for some $\beta \in \IC\ssm
  I'_n$. We deduce $g_n(\alpha_i) = g_n(h_n(\beta))=\beta$, which is a
  contradiction. So no $\alpha_i$ lies in $U$. This implies the rest of (i).
  
  We apply Theorem 5.2.3~\cite{Ransford} with $K_1
  =I'_n,K_2=\ecp\ssm U,D_1=\ecp\ssm I'_n,D_2= U,$ and $f$
  as ${h_n}|_{D_1}\colon {D_1}\rightarrow {D_2}$ extended as above to
  send $\infty\mapsto\infty$. Hence
  \begin{equation*}
    \mathrm{cap}(\ecp \ssm U) \le \mathrm{cap}(I'_n).
  \end{equation*}
  
  Recall that $I'_n$ is a hedgehog with at most $qd^n+m$ quills
  and lengths given in (\ref{def:In}) and (\ref{def:Iprimen}). 
  Dubinin's Theorem~\cite{Dubinin} implies
  \begin{equation*}
    \mathrm{cap}(I'_n) \le 4^{-1/(qd^n+m)}
    \max\{|\gamma_1|^{1/d^n},\ldots,|\gamma_q|^{1/d^n},
    |g_n(\alpha_i)| : \alpha_i\in\IC \ssm H_f^{(n)} \}. 
  \end{equation*}

  Combining these two estimates and using $|g_n(\alpha_i)| =
  |f^{(n)}(\alpha_i)|^{1/d^{n}}$ if $\alpha_i\in\IC\ssm H_f^{(n)}$ yields
  \begin{equation*}
    \mathrm{cap}(\ecp\ssm U) \le 4^{-1/(qd^n+m)}
    \max\{|\gamma_1|,\ldots,|\gamma_q|,
    |f^{(n)}(\alpha_1)|
    ,\ldots,|f^{(n)}(\alpha_m)|\}^{1/d^n}.
  \end{equation*}
  This completes the proof of (ii). 
% \end{proof}
% 
% \begin{proof}[Proof of Proposition~\ref{prop:constructK2}]
  % We apply (\ref{eq:cff}) to the $\alpha_i$ and observe
  % $|f^{(n)}(\alpha_i)|\le e^{\cf(f) + \lambda_f(f^{(n)}(\alpha_i))}
  % =e^{\cf(f) + d^n\lambda_f(\alpha_i)}$. This completes part (ii). 
\end{proof}

Next we make the estimate in Proposition~\ref{prop:constructK2} more explicit
for polynomials of the shape $T^d+c$. 

\begin{prop}
  \label{prop:constructK}
  Let $d\ge 2$ be an integer and let $f=T^d+c\in\IC[T]$.
  Let $n\in\IN$ such
  that $\pco_{n}^+(f)$ is contained in a hedgehog with at most $q\ge 0$
  quills.  
  Let $m\ge 1$ and let $\alpha_1,\ldots,\alpha_m \in\IC$.
  %and let $n\in\IN$.
  There is a simply connected domain $U\subset\ecp$
  with the following properties.
  \begin{enumerate}
  \item[(i)] We have   $0,\alpha_1,\ldots,\alpha_m\not\in U$ and
    $\ecp\ssm U$ is a compact subset of $\IC$.
  \item[(ii)] We have
    \begin{equation*}
      \label{eq:capKbound2}
      \mathrm{cap}(\ecp\ssm U)\le
      4^{-1/(qd^n+m)} \cf(f)^{1/d^n} e^{\max\{\lambda_f(0),\lambda_f(\alpha_1), \ldots,
      \lambda_f(\alpha_m)\}}.
    \end{equation*}
  \end{enumerate}
\end{prop}
\begin{proof}
  By hypothesis  there are $\gamma_1,\ldots,\gamma_q\in\IC^\times$ with
  $\pco_{n}^+(f)\subset \hh(\gamma_1,\ldots,\gamma_q)$.  By
  shortening and omitting quills we may assume that
  each $\gamma_i$ lies in $\pco_{n}^+(f)$.
  In other words, for each $i$ there is $n_i\in \{1,\ldots,n\}$ with
  $\gamma_i = f^{(n_i)}(0)$. 
  % Lemma~\ref{lem:iteratefabsbound}
  By the definition (\ref{def:Cf}) of $\cf(f)$ we have
  $|\gamma_i|\le \cf(f) e^{\lambda_f(f^{(n_i)}(0))} \le \cf(f) e^{d^{n_i}
    \lambda_f(0)} \le \cf(f) e^{d^n \lambda_f(0)}$.
  The same definition also gives $|f^{(n)}(\alpha_i)|\le \cf(f)
  e^{d^n\lambda_f(\alpha_i)}$ for all $i$. The proof follows from
  Proposition~\ref{prop:constructK2}.  
\end{proof}

Our applications require $\lambda_f(0)=0$ in (\ref{eq:capKbound2}). 

We conclude this section by deducing upper bounds for $\cf(f)$ when
$f=T^d+c\in\IC[T]$. Ingram has related estimates, see Section
2~\cite{Ingram:lowerbound}.

\begin{lemma}
  \label{lem:localhgtlb1}
  Let $d\ge 2$ be an integer and let $c\in \IC$.
  The polynomial $T^d -T- |c|$ has a unique root $r$ in 
  $(0,\infty)$.
  Then $r\ge 1$ and 
  $\log^+|z^d+c| \ge -(d-1)\log r + d\log^+ |z|$
  for all $z\in\IC$ with $|z|\ge r$. 
\end{lemma}
\begin{proof}
  Both statements are invariant under rotating $c$. So we may assume
  $c\in [0,\infty)$.
  
  The polynomial $T^d-T-c$ is a convex function on $(0,\infty)$ and
  it is negative for small
  positive arguments. So it has exact one root $r$ in
  $(0,\infty)$. We must have $r\ge 1$ as $1^d-1-c\le 0$. 

  Let $z\in\IC$. We prove the lemma by showing 
  \begin{equation}
    \label{eq:maxmaxlb}
    \frac{\max\{1,|z^d+c|\}}{\max\{1,|z|^d\}} \ge \frac{1}{r^{d-1}}
  \end{equation}
  for all $z\in\IC$ with $|z|\ge r$.
  Note that $|z|\ge r\ge 1$ gives  $|z|^d-c\ge |z|$.
  The triangle inequality implies $|z^d+c|\ge \bigl||z^d|-c\bigr|\ge
  |z|\ge 1$. So the left-hand side of (\ref{eq:maxmaxlb}) equals
  $|z^d+c|/|z^d|=|1+c/z^d|>0$.
  By the Minimum Modulus Principle, 
  we have $|1+c/z^d|\ge 1$ for all $z$ with $|z|\ge r$,
  or $|1+c/z^d|$ attains a minimum $<1$ on the boundary $|z|=r$.
  In the first case (\ref{eq:maxmaxlb}) holds true.
  Finally, if $|z|=r$ and $|1+c/z^d|$ is minimal, then necessarily
  $z^d =
  -r^d$ by elementary geometry. In this case
  $|1+c/z^d| = |1-c/r^d|=|(r^d-c)/r^d| = r^{1-d}$ and we
  recover (\ref{eq:maxmaxlb}). 
\end{proof}

\begin{lemma}
  \label{lem:iteratefabsbound}
  Let $d\ge 2$ be an integer and let $c\in\IC$.
  Let $r\ge 1$ be as in Lemma~\ref{lem:localhgtlb1}, then
  $\cf(T^d+c) \le r$.
\end{lemma}
\begin{proof}
  Set $f=T^d+c$. 
  Observe that by Lemma~\ref{lem:localhgtlb1} we have $|f(z)|\ge r$ if
  $|z|\ge r$. By a simple induction we find  $\log^+ |f^{(n)}(z)| \ge -(d-1)\log r +
  d\log^+ |f^{(n-1)}(z)|$ for all $n\ge 1$ if $|z|\ge r$. 
  Considering the telescoping sum we get
  \begin{alignat*}1
    \frac{\log^+|f^{(n)}(z)|}{d^n} - \log^+|z| &= \sum_{j=1}^n
    \frac{\log^+|f^{(j)}(z)|}{d^j} - \frac{d\log^+
      |f^{(j-1)}(z)|}{d^{j}} \\
    & \ge \sum_{j=1}^n \frac{-(d-1)\log r}{d^j}
    \ge  \sum_{j=1}^\infty \frac{-(d-1)\log r}{d^j} = -\log r.
  \end{alignat*}
  On taking the limit $n\rightarrow \infty$ the
  left-hand side becomes $\lambda_f(z)-\log^+|z|$ by (\ref{eq:deflambdafv}). 
  So we obtain $\log^+|z| \le \log r + \lambda_f(z)$ if $|z|\ge r$.
  Since $\lambda_f$ is non-negative,  this bound also holds if
  $|z|<r$.
  We obtain $\log^+|z|-\lambda_f(z)\le \log r$ for all $z\in\IC$. The
  lemma follows from
  the definition (\ref{def:Cf}).  
\end{proof}

\begin{lemma}
  \label{lem:Cflessthan2}
  Let $d\ge 2$ be an integer, let $c\in\IC$, and let $f=T^d+c$.
  Suppose $\pco^+(f)=\{f^{(n)}(0)
  : n\in\IN\}$ is a bounded set. 
  \begin{enumerate}
  \item [(i)] We have $|c|\le 2^{1/(d-1)}$ and
 %   \begin{equation}
%      \label{eq:CTsqrcbndper}
      $\cf(f) \le 2^{1/(d-1)}$. % Moreover, if $\cf(f) = 2^{1/(d-1)}$
      % then $|c|=2^{1/(d-1)}$. 
      % \cf(f) \le \frac 12 \left(1+ \left(1+\frac{4|c|}{d-1}\right)^{1/2}\right)
      % \le 1+\frac{2^{1/(d-1)}}{d-1}.
  %  \end{equation}
    \item[(ii)]
      Suppose $c^{d-1}\not=-2$, then $\cf(f) < 2^{1/(d-1)}$. 
%      If $f\not=T^2-2$, then  $\cf(f)<2$.
    \item[(iii)] If $d$ is odd, then $\cf(f) < 2^{1/(d-1)}$.
  \end{enumerate}
\end{lemma}
\begin{proof}
  By hypothesis we have $\lambda_f(c)=\lambda_f(f(0))=0$. Thus $|c|
  \le \cf(f)\le r$ by Lemma~\ref{lem:iteratefabsbound} where $r$ is the unique
  positive real number with $r^d=r+|c|$. As $T^d-T-|c|$ has no roots on
  $(0,r)$ by Lemma~\ref{lem:localhgtlb1} we find that it takes
  negative values on $(0,r)$. So $|c|^d \le |c|+|c|$ and therefore
  $|c|\le t$ where $t=2^{1/(d-1)}$. We note $t^d-t-|c|\ge t^d-t-t=0$
  and hence $t\ge r$. We conclude $\cf(f)\le r\le t = 2^{1/(d-1)}$.
  Part (i) follows. 
  
  For the proof of (ii) let us assume $\cf(f) = 2^{1/(d-1)}$ and
  retain the notation from the proof of part (i). Then $r=t$ and hence
  $2^{1/(d-1)}$ is a root of $T^d-T-|c|$. This yields $|c|=2^{1/(d-1)}$
  after a short calculation.

  If $d=2$, then it is well-known that the Mandelbrot set meets the
  circle of radius $2$ with center $0$ in the single point $-2$.  This
  implies (ii) for $d=2$.
  Here is a direct
  verification that extends to all $d\ge 2$. It involves
  $c^d+c$, the  second iterate
  of $0$  under $f$.  
  Indeed, its orbit under $f$ is also bounded and so 
  $|c||c^{d-1}+1|=|c^d+c|\le \cf(f)=2^{1/(d-1)}$.
  But $|c|=2^{1/(d-1)}$ and thus $|c^{d-1}+1|\le 1$.
  The circle around $0$ of radius $2$ meets the closed disk around $-1$ of
  radius $1$ in the single point $-2$. We conclude $c^{d-1}=-2$ and
  this implies (ii).

  In (iii) we have that $d$ is odd. We will prove  $c^{d-1}\not=-2$ by
  contradiction;
  (iii) then follows from (ii). Indeed,  if $c^{d-1}=-2$, then $f^{(2)}(0)=c^d+c=-c$. So
  $f^{(3)}(0) = -c^d+c$ has absolute value $|c||c^{d-1}-1| = 3|c| >
  2^{1/(d-1)}\ge \cf(f)$. But then $\lambda_f(f^{3}(0))>0$ and this
  contradicts the hypothesis that the $f$-orbit of $0$ is bounded. 
\end{proof}

We are now ready to prove Lemma~\ref{lem:quillhypothesis} from the introduction.

\begin{proof}[Proof of Lemma~\ref{lem:quillhypothesis}]
  Let $\sigma_0$ be a complex embedding of $K$ as in (i) or (ii).
  By hypothesis,
  $\sigma_0(f)\not=T^2-2$ and $d=p$ is a prime. Moreover,
  $\pco^+(\sigma_0(f))$ is bounded in both (i) and (ii). So
  Lemma~\ref{lem:Cflessthan2} yields $\cf(\sigma_0(f))<2^{1/(p-1)}\le
  2$; if $p=2$ we use part (ii) and if $p\ge 3$ we use (iii).
  
  For case (i) the embedding  $\sigma_0$  real. The
   set
  $\pco^+(\sigma_0(f))=\{\sigma_0(f)(0),\sigma_0(f)^{(2)}(0),\ldots\}$
  lies in $\IR$. So it is contained
  in the union $[-\alpha,0]\cup [0,\beta]$ of $q=2$ quills.
  Thus $q\log \cf(\sigma_0(f))<\log 4$ and (\ref{eq:quillcondition}) is satisfied.  

  In case (ii)  the set $\pco^{+}(\sigma_0(f))$ is finite and 
  can be covered  by 
  $q\le 2p-2$ quills.
  Now we find $q\log \cf(\sigma_0(f)) < 2(p-1) \log 2^{1/(p-1)} = \log
  4$ and (\ref{eq:quillcondition}) is again satisfied.  
\end{proof}

%%% Local Variables:
%%% TeX-master: "main"
%%% End:

\section{Proof of Main Results}
\label{sec:proofs}

The following lemma is well-known; the second claim goes back to
Gleason, see Lemme 2, Expos\'e XIX~\cite{DouadyHubbard:OrsayI} for
$p=2$, and to Epstein and Poonen~\cite{Epstein:Integrality}.
For the reader's convenience we give a proof. 

\begin{lemma}
  \label{lem:unramified}
  Let $p$ be a prime number and $f=T^p+c \in \IC[T]$.
  \begin{enumerate}
  \item [(i)] If $0$ is $f$-preperiodic, than $c$ is an
    algebraic integer.
  \item[(ii)] If $0$ is $f$-periodic, then $\IQ(c)/\IQ$ is unramified
    above $p$.
  \end{enumerate}
\end{lemma}
\begin{proof}
  We shall 
  consider $P=T^p+X$ as polynomial in $\IZ[T,X]$.

  We have $P^{(0)}=T,P^{(1)}=P=T^p+X, P^{(2)} = P(P(T,X),X)=(T^p+X)^p+X,
  P^{(3)}=P(P^{(2)}(T,X),X)=((T^p+X)^p+X)^p+X$ etc.
  A simple induction shows that $P^{(k)}$ is monic of degree $p^{k-1}$
  in $X$ for all $k\in\IN$. 

  Suppose $c\in\IC$ such that $0$ is $f$-preperiodic where $f=T^p+c$.
  So there exist integers $k,l$ with $0\le k<l$ such that
  $P^{(k)}(0,c)=P^{(l)}(0,c)$. This produces a monic
  polynomial in integral coefficients
  of degree $p^{l-1}$ that vanishes at $c$. Part (i) follows.
  
  For part (ii) we observe that
  $P^{(l)}=(P^{(l-1)})^p+X$ for $l\ge 1$. The derivative by $X$ satisfies
  $\frac{\partial }{\partial X} P^{(l)} = pP^{(l-1)}
  \frac{\partial }{\partial X}P^{(l-1)}+1 \in 1 + p\IZ[T,X]$.
  We specialize $T$ to $0$ and get 
  $\frac{\partial }{\partial X} P^{(l)}(0,X) \in 1 + p\IZ[X]$.

  The determinant of  the Sylvester matrix of the pair
  $P^{(l)}(0,X),\frac{\partial }{\partial X} P^{(l)}(0,X)$ 
  is up-to sign the discriminant of $P^{(l)}(0,X)$. The
  reduction modulo $p$ of this matrix has upper triangular form with
  diagonal entries $\equiv \pm 1 \imod p$. In particular, the
  discriminant is not divisible by $p$.

  If $0$ is $f$-periodic with $f=T^p+c$, then $f^{(l)}(0)=0$ for some
  $l\in\IN$. The $\IQ$-minimal polynomial of $c$ has integral
  coefficients by (i) and it divides $P^{(l)}(0,X)$. So its
  discriminant is not divisible by $p$ either. This implies (ii). 
\end{proof}

The following result is a special case of the P\'olya--Bertrandias
Theorem; it is crucial for our application and is also at the core of
Dimitrov's proof of the Schinzel--Zassenhaus Conjecture.

Let $K$ be a number field and let $\cO_K$ denote the ring of algebraic
integers of $K$. For $\phi = \sum_{j\ge 0} \phi_j X^{-j}\in K[[1/X]]$
and $\sigma\in\hom(K,\IC)$ we define $\sigma(\phi) = \sum_{j\ge 0}
\sigma(\phi_j)X^{-j}\in\IC[[1/X]]$.

\begin{thm}[P\'olya--Bertrandias]
  \label{thm:polyabertrandias}
  Let $K$ be a number field and $\phi\in\cO_K[[1/X]]$. Suppose
  that for each $\sigma\in
  \hom(K,\IC)$ there exists a connected open subset $U_\sigma
  \subset\IC$ with the following properties.
  \begin{itemize}
  \item[(i)] The complement $\IC\ssm U_\sigma$ is compact. 
  \item [(ii)] The formal power series
    $\sigma(\phi)$ converges at all $z\in\IC$ with $|z|$ sufficiently large.
  \item[(iii)] The holomorphic function induced by $\sigma(\phi)$ as in
    (ii) extends to a holomorphic function on $U_\sigma$. 
  \end{itemize}
  If $\prod_{\sigma\in\hom(K,\IC)} \mathrm{cap}(\IC\ssm U_\sigma)<1$, then $\phi$
  is a rational function. 
\end{thm}
\begin{proof} 
  This is a special case of  Theorem~5.4.6~\cite{Amice}. Note that our
  power series has integral coefficients, so we can take
  $P_1(K)=\emptyset$ in the reference.
  If $\overline\sigma$ denotes $\sigma$ composed with complex
  conjugation, then we may take $U_{\overline\sigma}$ to be the
  complex conjugate of $U_{\sigma}$. Recall that $\mathrm{cap}(\IC\ssm
  U_\sigma)$ is the transfinite diameter of $\IC\ssm U_\sigma$.   
\end{proof}

We come to the main technical result of our paper.

\begin{prop}
  \label{prop:prefinal}
  Let $p$ be a prime number, let $K$ be a number field,  and let
  $f=T^p+c\in K[T]$.
  Suppose further that there are $k,l\in\IN$ with $k<l,
  f^{(k-1)}(0)=f^{(l-1)}(0),$ and the following properties. 
  \begin{enumerate}
  \item [(i)]  % Suppose that $k,l\in\IN$  satisfy $k<l$ and
    % $f^{(k-1)}(0)=f^{(l-1)}(0)$.
    Let $\delta$ be as in  (\ref{def:delta}) and assume $a^{p^{k-\delta-1}}\equiv a^{p^{l-\delta-1}} \imod
    {p\cO_K}$ for all $a\in \cO_K$.
  \item[(ii)] For all field embeddings $\sigma\in \hom( K,\IC)$    
    we assume  that there exists $n_\sigma\in\IN$ such that  $\pco_{n_\sigma}^+(\sigma(f))$ is contained in a hedgehog with at
    most $q_\sigma\ge 0$ quills. 
  \end{enumerate}
  Let $x\in\overline K$ be an algebraic integer, then
  one of the following conclusions holds true:
  \begin{enumerate}
  \item [(A)]
    We have
    \begin{equation}
      \label{eq:lambdaprelb}
      \lambda^{\mathrm{max}}_{f}(x)
      \ge  \frac{1}{[K:\IQ]}\sum_{\sigma\in\hom(K,\IC)}
      \frac{1}{p^{l+n_{\sigma}}}\left(\frac{\log 4}{q_\sigma+2[K(x):K]p^{-n_\sigma}}-{\log
        \cf(\sigma(f))}\right)
    \end{equation}
  \item[(B)] We have $[K(f^{(k)}(x)):K]\le [K(x):K]/p$.
  \item[(C)] The point $x$ is $f$-preperiodic,  $\mathrm{preper}(x) \le k$, and
    $\mathrm{per}(x) \le (l-k)[K(x):K]$.
  \end{enumerate}
\end{prop}
\begin{proof}
  We begin by observing $c\in\cO_K$ by Lemma~\ref{lem:unramified}(i).  
  Let $A\in K[X]$ be the $K$-minimal polynomial of $x$. Then $A\in
  \cO_K[X]$ has degree $D=[K(x):K]$

  We apply Proposition~\ref{prop:powerseries} to the ring $R=\cO_K$. There
  exists $\phi \in \cO_K[[1/X]]$ with
  \begin{equation}
    \label{eq:phippminus1}
      \phi^p = (A_l/A_k)^{p-1}
  \end{equation}
  and
  \begin{equation}
    \label{eq:phiprops}
    \phi = \Phi_p(p^2 B/A_k^p)\quad\text{where}\quad\deg B \le \deg A_k^p -
    1 =pD-1;
  \end{equation}
  we recall (\ref{def:Phip}). 
  Then
  $\sigma(\phi)^p  = \sigma(A_l/A_k)^{p-1}$
  for all $\sigma\in\hom(K,\IC)$.  
  % The degree condition guarantees that $\sigma(\phi) = \Phi_p(p^2
  % \sigma(B)/\sigma(A_k)^p)\in\IC[[1/X]]$ converges for all $z\in\IC$ with $|z|$
  % sufficiently large. 

  We split $\sigma(A) = (X-x_{\sigma,1})\cdots(X-x_{\sigma,D})$ where
  $x_{\sigma,1},\ldots,x_{\sigma,D}\in\IC$ are conjugates of $\sigma(x)$
  over $\sigma(K)\subset\IC$. Then
  \begin{equation*}
    \sigma(\phi)^p =\prod_{j=1}^D
    \left(\frac{X-\sigma(f)^{(l)}(x_{\sigma,j})}{X-\sigma(f)^{(k)}(x_{\sigma,j})}\right)^{p-1}.
  \end{equation*}

  By the degree condition in (\ref{eq:phiprops}) the
  formal power series $\sigma(\phi) \in \IC[[1/X]]$ converges for all
  $z\in\IC$ for which $|z|$ is
  sufficiently large. For these $z$ we have
  \begin{equation}
    \label{eq:holomorphic1}
    \sigma(\phi)(z)^p    =\prod_{j=1}^D
    \left(\frac{1-\sigma(f)^{(l)}(x_{\sigma,j})/z}{1-\sigma(f)^{(k)}(x_{\sigma,j})/z}\right)^{p-1}. 
  \end{equation}

  We apply
  Proposition~\ref{prop:constructK} to $\sigma(f) = T^p+\sigma(c)$ and
  $n_\sigma$. By hypothesis,
  $\pco_{n_\sigma}^+(\sigma(f))$ % \subset\pco^+(\sigma(f))$
  is contained in a hedgehog with
  at most $q_\sigma\ge 0$ quills. We take $m=2D$ and set
  \begin{equation}
    \label{eq:choosealpha}
    \alpha_1 = \sigma(f)^{(k)}(x_{\sigma,1}),\ldots,
    \alpha_D = \sigma(f)^{(k)}(x_{\sigma,D}),
    \alpha_{D+1} = \sigma(f)^{(l)}(x_{\sigma,1}),\ldots,
    \alpha_{2D} = \sigma(f)^{(l)}(x_{\sigma,D}).
  \end{equation}
  We obtain a simply connected domain $U_\sigma\subset\ecp$ with the two stated
  properties.  

  The right-hand side of (\ref{eq:holomorphic1}) is well-defined and
  non-zero for all $z\in U_\sigma\ssm \{\infty\}$; indeed, none among
  $0$ and the (\ref{eq:choosealpha}) lie in $U_\sigma$ by
  Proposition~\ref{prop:constructK}. So the right-hand side extends to a
  holomorpic function on $U_\sigma$ that never vanishes; the extension
  maps $\infty$ to $1$. By the Monodromy Theorem from complex analysis
  and since $U_\sigma$ is
  simply connected we conclude that $z\mapsto \sigma(\phi)(z)$, defined \textit{a
    priori} only if $|z|$ is large, extends to a holomorphic map
  $\sigma(\phi)\colon U_\sigma\rightarrow\IC$.

  We split up into two cases.

  \textbf{Case 1.} First, suppose that $\phi$ is not a rational function. 

  Theorem~\ref{thm:polyabertrandias} implies
  \begin{equation}
    \label{eq:betrandiaslb}
    0 \le \sum_{\sigma\in\hom(K,\IC)}\log\mathrm{cap}(\ecp\ssm
    U_\sigma). 
  \end{equation}
  Recall that $\lambda_{\sigma(f)}(0)=0$ as $f$ is post-critically
  finite by hypothesis.
  This contribution can be omitted from the capacity bound from
  Proposition~\ref{prop:constructK} since the local canonical height is
  non-negative. So $k\le l$ implies
  \begin{equation}
    \label{eq:capacityub}
    \sum_{\sigma\in\hom(K,\IC)}\log \mathrm{cap}(\ecp\ssm
    U_\sigma) \le \sum_{\sigma\in \hom(K,\IC)}
    -\frac{\log 4}{q_\sigma p^{n_\sigma}+2D}+\frac{\log \cf(\sigma(f))}{p^{n_\sigma}}
    +p^l \max\{\lambda_{\sigma(f)}(x_{\sigma,1}),\ldots,\lambda_{\sigma(f)}(x_{\sigma,D})\}.
  \end{equation}
%  we used that the local canonical height is non-negative.
  % and
  % $\lambda_{\sigma(f)}(0)=0$. The latter condition follows since $f$ is
  % post-critically finite.

  We compare (\ref{eq:betrandiaslb}) and (\ref{eq:capacityub})
  and rearrange to find
  \begin{equation*}
    \sum_{\sigma\in \hom(K,\IC)}
    \max\{\lambda_{\sigma(f)}(x_{\sigma,1}),\ldots,\lambda_{\sigma(f)}(x_{\sigma,D})\}
    \ge \frac{1}{p^{l}}\sum_{\sigma\in\hom(K,\IC)}
    \left(\frac{\log 4}{p^{n_\sigma}q_\sigma+2D} -
      \frac{\log
      \cf(\sigma(f))}{p^{n_\sigma}}\right).
  \end{equation*}
  The left-hand side is at most $[K:\IQ]\lambda^{\mathrm{max}}_f(x)$. So
  conclusion (A) holds.

  \textbf{Case 2.} Second, suppose that $\phi$ is a rational function.

  In particular, $\phi$ is meromorphic on $\IC$.
  It follows from  (\ref{eq:phippminus1})  that
  any root of $A_l/A_k$ has vanishing order divisible by $p$. Let $F$
  be a splitting field of $A$. Then $\{f^{(k)}(\tau(x)) : \tau \in
  \mathrm{Gal}(F/K) \}$ are the  roots of $A_k$ and $\{f^{(l)}(\tau(x))
  : \tau \in \mathrm{Gal}(F/K) \}$ are the roots of $A_l$.

  We split up into two subcases.

  \textbf{Subcase 2a.} Suppose $\#\{f^{(k)}(\tau(x)) : \tau\in
  \mathrm{Gal}(F/K) \} \le D/p$. This means
  $[K(f^{(k)}(x)):K]\le D/p$ and we are in conclusion (B) because $D=[K(x):K]$. 

  \textbf{Subcase 2b.}
  Suppose $\#\{f^{(k)}(\tau(x)) : \tau\in \mathrm{Gal}(F/K) \} >
  D/p$. Recall $D=\deg A_k$. Then there exists $\tau$ such that the
  vanishing order of $A_k$ at $f^{(k)}(\tau(x))$ is positive and strictly less than
  $p$. But then it must also be a zero of $A_l$. After replacing
  $\tau$ we may assume
  $f^{(k)}(\tau(x))=f^{(l)}(x)$.  Recall $k<l$. A simple induction
  shows $\tau^e (f^{(k)}(x)) = f^{((l-k)e+k)}(x)$; indeed for $e\ge 2$
  we have
  $$\tau^e(f^{(k)}(x))= \tau( \tau^{e-1}(f^{(k)}(x)))
  =\tau(f^{((l-k)(e-1)+k)}(x)) 
  =f^{((l-k)(e-1)+k)}(\tau(x))=f^{((l-k)e+k)}(x)$$ as $f\in K[T]$.

  Take $e$ to be the minimal positive integer with $\tau^{e} \in
  \mathrm{Gal}(F/K(x))$. So $\tau^e(f^{(k)}(x))
  =  f^{(k)}(x)$. We conclude
  \begin{equation*}
    f^{(k)}(x) = f^{((l-k)e+k)}(x). 
  \end{equation*}
  In particular, $x$ is $f$-preperiodic and $f^{(k)}(x)$ is
  $f$-periodic. In other words, $\mathrm{preper}(x)\le k$. 

  By the Pigeonhole Principle there are integers $e$ and $e'$ with $0\le
  e<e'\le [K(x):K]$ with $\tau^{e}(f^{(k)}(x)) = \tau^{e'}(f^{(k)}(x))$. So
  $f^{((l-k)e+k)}(x)=f^{((l-k)e'+k)}(x)$. Hence $x$ has minimal period at most
  $(l-k)(e'-e) \le (l-k)[K(x):K]$. We are in conclusion (C).
\end{proof}

\begin{thm}
  \label{thm:preperiodic2}
  Let $p$ be a prime number, let $K$ be a number field,  and let
  $f=T^p+c\in K[T]$.
  Suppose further that there are $k,l\in\IN$ with $k<l,
  f^{(k-1)}(0)=f^{(l-1)}(0),$ and the following properties. 
  \begin{enumerate}
  \item [(i)]
    Let $\delta$ be as in (\ref{def:delta})
    and assume $a^{p^{k-\delta-1}}\equiv a^{p^{l-\delta-1}} \imod
    {p\cO_K}$ for all $a\in \cO_K$.
  \item[(ii)]
    We suppose hat $f$ satisfies the Quill Hypothesis. 
  \end{enumerate}
  % There exists a field embedding $\sigma_0\in\hom(K,\IC)$
  %   such that $\pco^+(\sigma_0(f))$ is contained in a hedgehog with at
  %   most $q\ge 0$ quills.    
  % \end{enumerate}
  % If
  % \begin{equation}
  %   \label{eq:preperiodicthmcondition}
  %   q \log \cf(\sigma_0(f)) < \log 4,
  % \end{equation}
  Then there exists a constant $\kappa>0$ depending on the data above with the following properties.
  \begin{enumerate}
  \item [(i)] If $x$ is an algebraic integer and $f$-wandering, then
    $\lambda_f^{\mathrm{max}} (x) \ge \kappa/[K(x):K]^k$. 
  \item[(ii)] If $x\in\overline K$ is $f$-wandering, then
    $\hat h_f (x) \ge \kappa/[K(x):K]^{k+1}$. 
  \end{enumerate}
\end{thm}
\begin{proof}
  We remark that $c\in \cO_K$ by Lemma~\ref{lem:unramified}(i).
  
  Our proof of (i) is by induction on $[K(x):K]$. 
  We will apply Proposition~\ref{prop:prefinal}.  
  Observe that an $f$-wandering point cannot lead to conclusion (C).

  Let $\sigma_0\in\hom(K,\IC)$ and $q\ge 0$ be as in the Quill
  Hypothesis. We introduce two parameters, $\epsilon_0,\epsilon \in
  (0,1)$ both are sufficiently small and fixed in terms of $f$ and the
  other given data such as $q$ and $\sigma_0(f)$, but independent of $x$.
  We will fix $\epsilon$ in function of $\epsilon_0$.
  
  Let $n_{\sigma_0}\ge 1$ be the unique integer with
  $p^{n_{\sigma_0}} \ge [K(x):K]/\epsilon_0 > p^{n_{\sigma_0}-1}$.
  For all $\sigma\in\hom(K,\IC)$ with $\sigma\not=\sigma_0$
  we fix the unique $n_{\sigma}\in\IN$ with
  $p^{n_{\sigma}} \ge [K(x):K]/\epsilon > p^{n_{\sigma}-1}$.
  
  For any $\sigma$ we set $q_\sigma=\#\pco^+(\sigma(f)) < \infty$.
  Then 
  $\pco^+(\sigma(f))$  is contained
  in a hedgehog with at most  $q_\sigma$  quills and therefore so is
  $\pco^+_{n_{\sigma}}(\sigma(f)) \subset \pco^+(\sigma(f))$.

  For $\epsilon$ sufficiently
  small we have
  \begin{alignat*}1
    Y=\sum_{\sigma\in\hom(K,\IC)}
    &\frac{1}{p^{n_\sigma}}\left(\frac{\log 4}{q_\sigma +
        2[K(x):K]p^{-n_\sigma}} -\log \cf(\sigma(f))\right)
    \\
    &\ge
    \frac{1}{p^{n_{\sigma_0}}}\left(
      \frac{\log 4}{q +
        2[K(x):K]p^{-n_{\sigma_0}}} -{\log \cf(\sigma_0(f))}\right)
    - \sum_{\sigma\not=\sigma_0}
    \frac{\log \cf(\sigma(f))}{p^{n_\sigma}}
    \\
    &\ge
    \frac{1}{p^{n_{\sigma_0}}}\left(
      \frac{\log 4}{q +
        2\epsilon_0} -{\log \cf(\sigma_0(f))}\right)
    - \frac{\epsilon}{[K(x):K]}\sum_{\sigma\not=\sigma_0}
    \log \cf(\sigma(f)).
  \end{alignat*}
  For $\epsilon_0$ small enough and fixed in function of $f$, the
  Quill Hypothesis (\ref{eq:quillcondition})  
  implies
  $(\log 4)/(q+2\epsilon_0) - \log\cf(\sigma_0(f)) \ge
  \kappa_1$ where $\kappa_1=\kappa_1(f) > 0$ depends only on $f$. So 
  \begin{equation*}
    Y \ge \frac{\kappa_1}{p^{n_{\sigma_0}}}     - \frac{\epsilon}{[K(x):K]}\sum_{\sigma\not=\sigma_0}
    \log \cf(\sigma(f))
    \ge
    \frac{1}{ [K(x):K]}\left(\frac{\kappa_1 \epsilon_0}{p}     - {\epsilon}\sum_{\sigma\not=\sigma_0}
    \log \cf(\sigma(f))\right).
  \end{equation*}
  Now we fix $\epsilon$ small in terms of $\epsilon_0$ and $f$ to
  achieve
  $\kappa_1\epsilon_0/(2p) \ge \epsilon \sum_{\sigma\not=\sigma_0}
  \log\cf(\sigma(f))$. So
  \begin{equation*}
    Y\ge \frac{\kappa_2}{[K(x):K]}\quad\text{with}\quad
    \kappa_2 = \frac{\kappa_1\epsilon_0}{2p}. 
  \end{equation*}
  
  If we are in conclusion (A) of Proposition~\ref{prop:prefinal}, then
  $\lambda_f^{\mathrm{max}}(x) \ge Y [K:\IQ]^{-1}p^{-l} \ge \kappa_2p^{-l}/[K(x):\IQ]$.
  This is stronger than the claim as we may assume $\kappa\le \kappa_2 [K:\IQ]^{-1}p^{-l}$.
    
  If we are in conclusion (B), then we have $[K(f^{(k)}(x)):K] \le
  [K(x):K]/p<[K(x):K]$. In particular, this rules out $x\in K$, so
  the base case of the induction was handled by conclusion (A) above.
  Now $\lambda_f^{\mathrm{max}}(f^{(k)}(x)) = p^k
  \lambda_f^{\mathrm{max}}(x)$. Induction on $[K(x):K]$ yields
  $p^k \lambda_f^{\mathrm{max}}(x) \ge \kappa /[K(f^{(k)}(x)):K]^k \ge \kappa
  p^k/[K(x):K]^k$.
  This completes the proof of (i) as conclusion (C) is impossible in
  the wandering case.

  For (ii) we observe that $\hat h_f(x)$ is a normalized
  sum of local canonical heights as in
  (\ref{def:csheight}). Moreover, every local canonical height
  takes 
  non-negative values. In particular, $\hat h_f(x) \ge
  \lambda_{f}^{\mathrm{max}}(x)/[\IQ(x):\IQ]$. If $x$ is an algebraic
  integer, then part (i) implies the desired lower bound for $\hat
  h_f(x)$ after adjusting $\kappa$.
  So assume that $x$ is not an algebraic integer. Let $F$ be a number
  field containing $x$. There is a non-Archimedean  $v\in M_{\IQ}$, which we may
  identify with a prime number, and a
  field embedding $\sigma\in\hom( F, \IC_v)$ with
  $|\sigma(x)|_v>1$.
  By anticipating ramification we find
  $|\sigma(x)|_v \ge v^{1/[\IQ(x):\IQ]}$. 
  Thus $|\sigma(x^p+c)|_v=|\sigma(x)|_v^p>1$
  by the ultrametric triangle inequality and since $c\in\cO_K$. 
  Furthermore $|\sigma(f^{(n)}(x))|_v =
  |\sigma(x)|_v^{p^n}$
  for all $n\in\IN$.
  Therefore, $\lambda_{\sigma(f),v}(\sigma(x)) = \log^+|\sigma(x)|_v
  \ge (\log v)/[\IQ(x):\IQ]\ge (\log 2)/[\IQ(x):\IQ]$.
  Again we use that local canonical heights are non-negative and conclude
  $\hat h_f(x) \ge (\log 2)/[\IQ(x):\IQ]^2\ge (\log 2)/[K(x):\IQ]^2 =
  [K:\IQ]^{-2}(\log 2)/[K(x):K]^2$. The theorem follows
  as we may assume $\kappa \le [K:\IQ]^{-2}\log 2$. 
\end{proof}

We consider an example before moving on. Let $c=-1.543689\ldots$ be
the real root of $T^3+2T^2+2T+2$. Then $[K:\IQ]=3$ where $K=\IQ(c)$.
Moreover, $f^{(3)}(0)=f^{(4)}(0)$ for $f=T^2+c$ and
$\pco^+(f)\subset\IR$ is contained in a union of $2$ quills. Let $\wp$
be a prime ideal of $\cO_K$ containing $2$, then $\wp^3= 2\cO_K$
follows from a pari/gp computation. We claim that $a^4 \equiv a^8
\imod {2\cO_K}$ for all $a\in\cO_K$. Indeed, $a^4(a^4-1)\in 2\cO_K$ if
$a\in \wp$. If $a\not\in\wp$, then $a$ is a unit modulo $2\cO_K$. So
$a^4 -1 \in 2\cO_K$ as the unit group $(\cO_K/2\cO_K)^\times$ has
order $4$. So hypothesis (i) in Theorem~\ref{thm:preperiodic2} is
satisfied with $(k,l)=(4,5)$; note $\delta=1$. The Quill Hypothesis is
met since $\cf(f)<2$ by Lemma~\ref{lem:Cflessthan2}(ii).

\begin{thm}
  \label{thm:canonicalhgtlb}
  Let $p$ be a prime number, let $K$ be a number field,  and let
  $f=T^p+c\in K[T]$. Suppose $k$ and $l$ are as in
  Theorem~\ref{thm:preperiodic2} and suppose that $f$ satisfies the
  Quill Hypothesis. 
  If $x\in\overline K$ is an $f$-preperiodic point, then
  \begin{equation}
    \label{eq:degreelbpreper}
    [K(x):K] \ge p^{\mathrm{preper}(x)/k-1}
  \end{equation}
  and
  \begin{equation}
    \label{eq:degreelbper}
    [K(x):K] \ge [K(f^{(\mathrm{preper}(x))}(x)):K] \ge \frac{\mathrm{per}(x)}{l-k}.
  \end{equation}
\end{thm}
\begin{proof}
  By hypothesis, the $f$-orbit of $0$ is finite and so is $\pco^+(f)$. 
  By Lemma~\ref{lem:unramified}(i) we see that $c$ is an algebraic
  integer. So $x$, an $f$-preperiodic point, is also an algebraic
  integer. Moreover, $\lambda_f^{\mathrm{max}}(x)=0$.

  Let $\sigma_0\in \hom(K,\IC)$ be as in the Quill Hypothesis. We fix
  $n_{\sigma_0}$ and $n_{\sigma}$ for $\sigma_0\not=\sigma\in
  \hom(K,\IC)$ in terms of $x$ as in the proof of
  Theorem~\ref{thm:preperiodic2} to ensure that right-hand side of
  (\ref{eq:lambdaprelb}) is strictly positive. We can rule out
  conclusion (A) of Proposition~\ref{prop:prefinal}.
  
  Then we are  either in conclusion (B) where 
  $[K(f^{(k)}(x)):K]\le [K(x):K]/p$ or conclusion (C) where
  $\mathrm{preper}(x) \le k$ and $\mathrm{per}(x) \le (l-k)[K(x):K]$.

  Observe that we have $\mathrm{preper}(f(x)) =
  \max\{0,\mathrm{preper}(x)-1\}$. % and $\mathrm{per}(f(x)) =
  % \mathrm{per}(x)$. 
  Iterating gives $\mathrm{preper}(f^{(m)}(x)) =
  \max\{0,\mathrm{preper}(x)-m\}$ for all integers $m\ge 0$.

  We prove (\ref{eq:degreelbpreper}) by induction on $\mathrm{preper}(x)$.
  The claim is trivial if $\mathrm{preper}(x)\le k$. So let us assume
  $\mathrm{preper}(x)>k$. Then we are in conclusion (B) of
  Proposition~\ref{prop:prefinal}. Hence
  $[K(f^{(k)}(x)):K]\le [K(x):K]/p$ and
  $\mathrm{preper}(f^{(k)}(x)) = \mathrm{preper}(x)-k\ge 0$.
  By induction
  \begin{equation*}
    [K(x):K]\ge p [K(f^{(k)}(x)):K] \ge
    p^{\mathrm{preper}(f^{(k)}(x))/k}
    =p^{\mathrm{preper}(x)/k-1},
  \end{equation*}
  as desired.

  The first inequality in (\ref{eq:degreelbper}) is immediate. 
  We prove the second inequality in the case where $x$ is $f$-periodic
  first. 
  Indeed, then $K(f(x))=K(x)$ and so we are again in conclusion (C). 
  Hence $[K(x):K]\ge \mathrm{per}(x)/(l-k)$, as desired. 
  If $x$ is $f$-preperiodic then 
  (\ref{eq:degreelbper}) is applicable to the
  $f$-periodic 
$f^{(\mathrm{preper}(x))}(x)$. We find
  $[K(f^{(\mathrm{preper}(x))}(x)):K] \ge {\mathrm{per}(f^{(\mathrm{preper}(x))}(x))}/(l-k)$.  
  The theorem follows as $\mathrm{per}(f^{(m)}(x)) = \mathrm{per}(x)$
  for all integers $m\ge 0$.
\end{proof}

\begin{proof}[Proof of Theorem~\ref{thm:periodichgtlb}]
  We suppose $f^{(l-1)}(0)=0$ with $l\ge 2$.
  We want to apply Theorem~\ref{thm:preperiodic2} with $K=\IQ(c)$.
  Indeed, it suffices to verify that hypothesis (i) holds true for
  $k=1$ after possibly adjusting $l$.
  
  By Lemma~\ref{lem:unramified} the parameter $c$ is an algebraic integer and
  $K/\IQ$ is unramified above $p$.

  % The condition  verify $a\equiv a^{p^{l-1}} \imod {p\cO_K}$  for all
  % $a\in\cO_K$.
  We write $m$ for the least common multiple of $l-1$ and
  all residue degrees of all 
  prime ideals of $\cO_K$ containing $p$.
  So 
  $a\equiv a^{p^{m}} \imod {p\cO_K}$ for all $a\in\cO_K$.
  Moreover  $f^{(l-1)}(0)=0$ implies $f^{(m)}(0)=0$.
  Theorem~\ref{thm:periodichgtlb} 
  follows from Theorem~\ref{thm:preperiodic2} applied to $(k,l)=(1,m+1)$; note $\delta=0$.
\end{proof}

\begin{proof}[Proof of Corollary~\ref{cor:irrfactorgrowth}]
  We know that $c$ is
  an algebraic integer by Lemma~\ref{lem:unramified}(i).

  Let $x\in\overline K$ with $f^{(n)}(x)=y$.
  
  % We suppose first that $y$ is an $f$-preperiodic point. 
  As in the
  proof of Theorem~\ref{thm:periodichgtlb} we may assume that there
  exists $l\ge 2$ such that the pair $(1,l)$ satisfies the hypothesis of
  Theorem~\ref{thm:preperiodic2}.

  If $y$ is $f$-periodic and suppose $n\ge 0$ is minimal with $f^{(n)}(x)=y$. 
  To prove (i) we may assume  $n\ge m=\mathrm{per}(y)$.
  If $\mathrm{preper}(x) \le n-m$, then $f^{(n-m)}(x)$ is $f$-periodic
  and one among $f^{(n-m)}(x),\ldots,f^{(n-1)}(x)$ must equal $y$.
  But this contradicts the minimality of $n$
  and so $\mathrm{preper}(x) \ge n-m+1$.
  % Then
  % $\mathrm{preper}(f^{(n-m)}(x))\ge
  % 1$ by the minimality of $n$.
  % So $\mathrm{preper}(f^{(n-m)}(x))
  % = \max\{0,\mathrm{preper}(x)-n+m\}$ implies 
  % $\mathrm{preper}(x)=\mathrm{preper}(f^{(n-m)}(x)) + n-m \ge 1+n-m$.
  Theorem~\ref{thm:canonicalhgtlb} implies $[K(x):K]\ge
  p^{\mathrm{preper}(x)-1} \ge p^{n-m}$. We conclude (i). 

  The proof of  part (ii) is very similar. % Suppose $n\in\IN$ and
  % $f^{(n)}(x)=y$.
  %Then $\mathrm{preper}(y)\ge 1$ by hypothesis.
  We use 
  $1\le \mathrm{preper}(y)=\mathrm{preper}(f^{(n)}(x)) =
  \max\{0,\mathrm{preper}(x)-n\}$ to infer
  $\mathrm{preper}(x)=n+\mathrm{preper}(y)\ge n+1$.
  Theorem~\ref{thm:canonicalhgtlb} implies $[K(x):K]\ge
  p^{\mathrm{preper}(x)-1} \ge p^n$. But $[K(x):K]\le \deg
  (f^{(n)}-y)=p^n$, so part (ii) follows.

  For part (iii) let $x$ denote a root of $f^{(n)}-y$.
  We first suppose that $y$ is an algebraic integer. Then $x$ is also
  an algebraic integer.
%  A root  $x$ of $f^{(n)}-y$ is also an algebraic integer.
  Then $\lambda_f^{\mathrm{max}}(x) \ge \kappa_1/[K(x):K]$ by
  Theorem~\ref{thm:periodichgtlb}(i). The functional equation of the local
  canonical height implies $\lambda_f^{\mathrm{max}}(x) =
  \lambda_f^{\mathrm{max}}(y)/p^n$. We combine and rearrange to find
  $[K(x):K]\lambda_f^{\mathrm{max}}(y) \ge \kappa_1p^n$.
  % Since $y$ is an
  % algebraic integer, the only contribution to $\hat h_f(y)$ comes from
  % the infinite places. But $\hat h_f(y)>0$ as $y$ is an $f$-wandering
  % point.
  So $\lambda_f^{\mathrm{max}}(y)>0$ and $[K(x):K]\ge
  \kappa p^n$ with $\kappa = \kappa_1/\lambda_f^{\mathrm{max}}(y)$.

  Second, suppose that $y$ is not an algebraic integer. Then there
  exists a prime number $v$ and a field embedding $\sigma\in\hom(
  K,\IC_v)$ with $|\sigma(y)|_v>1$. 
  %Theorem~\ref{thm:preperiodic2}
  The ultrametric triangle inequality
  implies $|\sigma(x)|_v>1$ and $|\sigma(y)|_v=|\sigma(f^{(n)}(x))|_v =
  |\sigma(x)|_v^{p^n}$. So the ramification index of $K(x)/K$ at some
  prime ideal above $v$ grows like a positive multiple of $p^n$.
  In particular, $[K(x):K]\ge \kappa p^n$ for all $n\in\IN$ where
  $\kappa>0$ is independent of $n$. 
\end{proof}

\begin{proof}[Proof of Corollary~\ref{cor:wanderingminus1}]
  We will use Proposition~\ref{prop:prefinal} with $p=2$ and $K=\IQ$.
  
  Observe that
  $0=f^{(2)}(0)$. As $a\equiv a^4 \imod 2$ for all $a\in\IZ$ we see
  that  $(k,l)=(1,3)$ satisfies
  hypothesis  (i) of
  Proposition~\ref{prop:prefinal}; here $\delta=0$. %  and  the second hypothesis
  % is also satisfied  as $a\equiv a^4 \imod 2$ for all $a\in\IZ$. So we may work
  % with $(k,l)=(1,3)$. 
  
  The single quill $[-1,0]$ suffices to cover
  $\pco^+(f)=\{-1,0\}$. Moreover, the value $r$ from
  Lemma~\ref{lem:localhgtlb1} applied to $d=2$ and $c=-1$
  is the golden ratio $(1+\sqrt 5)/2$.
  So $\cf(T^2-1)\le (1+\sqrt 5)/2$ by
  Lemma~\ref{lem:iteratefabsbound}.  
  
  The proof of part (i) is by
  induction on $[\IQ(x):\IQ]$. We fix $n\in\IN$ to be minimal with
  $2^n \ge 3[\IQ(x):\IQ]$, so $2^{n-1}< 3[\IQ(x):\IQ]$.
  
  Observe that the $f$-wandering point $x$ is in conclusion (A) or
  (B) of Proposition~\ref{prop:prefinal}.
  If $x\in\IQ$ then we must be in conclusion (A).

  Conclusion (A) implies
  \begin{alignat*}1
    \lambda_f^{\mathrm{max}}(x) &\ge \frac{1}{2^{3+n}} \left(\frac{\log
        4}{1+2[\IQ(x):\IQ]2^{-n}} - \log \frac{\sqrt 5+1}{2}\right)
    \ge \frac{1}{2^{3+n}}\log\left(\frac{4^{3/5}\cdot 2}{\sqrt
        5+1}\right)
    \\
    &\ge \frac{1}{48}\log\left(\frac{4^{3/5}\cdot 2}{\sqrt 5+1}\right)\frac{1}{[\IQ(x):\IQ]}.
  \end{alignat*}
  Part (i) follows in this case and in particular if $x\in\IQ$; which
  is the  base case.

  In conclusion (B) we have $[\IQ(f(x)):\IQ]\le [\IQ(x):\IQ]/2$.
  Part (i) follows by induction on $[\IQ(x):\IQ]$ and since
  $\lambda_f^{\mathrm{max}}(f(x))=2\lambda_f^{\mathrm{max}}(x)$.

  The estimate in part (ii)  follows from
  part (i) in the integral case. In the non-integral case we
  use the same argument 
  as in the proof
  of Theorem~\ref{thm:preperiodic2}(ii) and conclude using 
   $\log 2 > \log(4^{11/10}/(\sqrt 5+1))/48$. 
\end{proof}

\begin{proof}[Proof of Corollary~\ref{cor:preperminus1}]
  As in the proof of Corollary~\ref{cor:wanderingminus1} we set
  $K=\IQ$ and 
  $(k,l)=(1,3)$. The single quill $[-1,0]$ covers $\pco^{+}(T^2-1)$ and
  we have $\cf(T^2-1)\le  (1+\sqrt{5})/2$.
  The numerical condition in the Quill Hypothesis
  %(\ref{eq:preperiodicthmcondition2})
  follows as $(1+\sqrt 5)/2<4$.
  The corollary follows from Theorem~\ref{thm:canonicalhgtlb}. 
\end{proof}
% The minimal period is at
% most $(l-k)m\le (l-k)[K(x):K]!$ and the preperiod is at most $k$.

% Then $\lambda_{\sigma(f),v}(\sigma(x))=0$ for all $\sigma\colon
% K(x)\rightarrow \IQ_v$ if $v$ is a finite place of $\IQ$.

% Let $\lambda = \max \{\lambda_{\sigma(f),\infty}(\sigma(x)) :
%   \sigma\colon K(x)\rightarrow\IC\}$. 

%%% Local Variables:
%%% TeX-master: "main"
%%% End:

\bibliographystyle{alpha}
\bibliography{literature}

\end{document}